\documentclass[12pt]{amsart}

\usepackage{amsmath}
\usepackage{amsfonts}
\usepackage{amssymb,enumerate}
\usepackage{physics}
\usepackage{amsthm}
\usepackage{fancyhdr}
\usepackage{tikz}
\usepackage{verbatim}
\usetikzlibrary{arrows,matrix}
\usepackage{hyperref}
\newcommand{\vect}[1]{\boldsymbol{#1}}
\usepackage{caption}
\usepackage[noend]{algpseudocode}
\usepackage{epstopdf}

\newtheorem{lemma}{Lemma}
\newtheorem{corollary}[lemma]{Corollary}
\newtheorem{proposition}[lemma]{Proposition}
\newtheorem{theorem}[lemma]{Theorem}
\newtheorem{question}[lemma]{Question}
\newtheorem{example}[lemma]{Example}

\newtheorem{remark}[lemma]{Remark}

\title[]{Lacunary Polynomial Compositions}



\author{A. Moscariello}
\address{Dipartimento di Matematica, Università di Pisa, Largo Bruno Pontecorvo 5, 56127 Pisa, Italy.}

\email{moscariello@mail.dm.unipi.it}

\begin{document}
\maketitle
\begin{abstract}
This work is a study of polynomial compositions having a fixed number of terms. We outline a recursive method to describe these characterizations, give some particular results and discuss the general case. In the final sections, some applications to Universal Hilbert Sets generated by closed forms of linear recurrence relations and to integer perfect powers having few digits in their representation in a given scale $x \ge 2$ are provided.
\end{abstract}

\section*{Introduction}
	A \emph{lacunary polynomial} (also called \emph{sparse polynomial}) is a polynomial where the number of terms is assumed fixed, with no control on the value of the degrees and coefficients of said terms. For instance, we may write $g(X)=a_1X^{l_1}+\dots+a_kX^{l_k}$ for a lacunary polynomial with at most $k$ different terms, where $k$ is fixed and $a_1,\ldots,a_k$ and $l_1,\ldots,l_k$ can be taken with no restriction. This assumption on the number of terms of a polynomial is equivalent to a bound on the polynomial's complexity. For instance, a non-constant monomial only admits one root (that is, zero), while the set of roots of a given polynomial having exactly two terms has a very simple structure, since these roots can only differ by some $n$th roots of unit.

This focus on the number of terms of a polynomial rather than the values of these terms or their degree, naturally brings some open questions, regarding in particular the behaviour of these polynomials under composition. For instance, Erd\H{o}s and R\'enyi independently conjectured (see \cite{E1}, \cite{R1}) the existence of a bound for the number of terms of the square of a given polynomial having a fixed number $k$ of terms, depending only on $k$. Later, this was proved by Schinzel (see \cite{S2}) in a more general setting, providing a lower bound for the number of terms of a lacunary power $P(T)^d$, where $P(T)$ is a lacunary polynomial. In the same work, he further conjectured that a bound could be found for the number of terms of the polynomial composition $f(g(X))$, depending on the number of terms of both polynomials involved. Schinzel's conjecture has then been proved by Zannier (\cite{Z1}), developing completely different methods from those used by Schinzel.

In this work, we will further extend this setting by studying lacunary polynomials obtained as composition of a Laurent polynomial $g(X_1,\ldots,X_\sigma)$ in $\sigma \ge 1$ indeterminates, with a classic univariate polynomial $f(T)$. Our motivation stems from some applications of this question in some arithmetic contexts. For this purpose, after an overview on this problem, we will focus on some special cases, according to our needs. The general Question is the following:
\begin{question}\label{mainQ}
	Let $\sigma \ge 1$ and $\rho \ge 0$ be integers, let $g(X_1,\ldots,X_\sigma) \in \mathbb{C}[X_1^{\pm 1},\ldots,X_\sigma^{\pm 1}]$ be a Laurent polynomial in the indeterminates $\vect{X} =(X_1,\ldots,X_\sigma)$, and let $f(T) \in \mathbb{C}[T]$, which we can assume monic without loss of generality. \\
	Determine for which polynomials $f$ and $g$ the identity
	\begin{equation}\label{mainEQ}
	f(g(\vect{X}))=a_1X_1^{l_1}+\ldots+a_\sigma X_\sigma^{l_\sigma}+a_{\sigma+1}T_1(\vect{X})+\ldots+a_{\sigma+\rho}T_\rho(\vect{X})
	\end{equation}
	holds, for $l_1,\ldots,l_\sigma$ positive integers, $a_1,\ldots,a_{\sigma+\rho} \in \mathbb{C}$ and $T_1,\ldots,T_\rho \in \mathbb{C}[X_1^{\pm 1},\ldots,X_\sigma^{\pm 1}]$ monomials in $X_1,\ldots,X_\sigma$. 
\end{question}

Notice that, while this formulation may appear very specific at first sight, it is actually quite generic. In order to see this, consider a polynomial composition of the form $$f(g(X_1,\ldots,X_\sigma))=a_1T_1(X_1,\dots,X_\sigma)+\ldots+a_kT_k(X_1,\ldots,X_\sigma),$$ where $T_1,\ldots,T_k$ are monomials in $X_1,\ldots,X_\sigma$. Among those monomials, we can choose a set $\mathcal{M}$ of multiplicative independent terms with maximal cardinality $\sigma$, and then choose a new set of indeterminates $Y_1,\ldots,Y_\sigma$ such that each element of $\mathcal{M}$ becomes a power $Y_i^{l_i}$, where the exponents $l_i$ are chosen so that all exponents of the original composition are still in $\mathbb{Z}$. With this substitution, it is very easy to check that a generic polynomial composition can be associated to an identity of the form (\ref{mainEQ}). 

We study some particular cases of this question, starting with lacunary polynomial powers ($f(T)=T^m$, Section 1) and then, leveraging the results obtained, we will investigate the general case, considering polynomial compositions with few monomials $T_i(X_1,\ldots,X_\sigma)$ (solving the cases $\rho \in {1, 2}$, Section 3). This choice is motivated by an application to a question concerning closed forms of linear recurrence relations and Universal Hilbert Sets, which we will first describe in Section 2. Also, in the Appendix, we provide a brief application of our results on lacunary polynomial powers to  perfect powers having few non-zero digits in their representation in a fixed scale, coming from a work of Corvaja and Zannier (\cite{CZ1}).

\section{Polynomial powers with few terms}
Consider a lacunary polynomial power $P(T)^d$ having exactly $k$ non-zero terms. We can assume without loss of generality that the term of degree zero is $1$, i.e. $$P(T)^d=1+ \sum_{i=1}^{k-1} \xi_iT^{l_i}.$$
The following remark will allow us to make some further assumptions in our study. 

\begin{remark}\label{2}
	\begin{enumerate}
		\item If $P(T)^d=1+\displaystyle\sum_{i=1}^{k-1} \xi_iT^{l_i}$, take the polynomial $Q(T)=\sqrt[d]{\frac{1}{\xi_{k-1}}T^{l_{k-1}}}P(T^{-1})$. Then $Q(T)^d=1+\left(\displaystyle\sum_{i=1}^{k-2}\frac{\xi_i}{\xi_{k-1}}T^{l_{k-1}-l_i}\right)+\frac{1}{\xi_{k-1}}T^{l_{k-1}}$. Rearranging the indices appearing in this equation in a way such that the exponents of $Q(T)^d$ are in increasing order, we see that the $\left \lceil \frac{k-1}{2} \right \rceil$th exponent of $Q(T)^d$ is $l_{k-1}-l_{k-1-\lceil\frac{k-1}{2}\rceil} \ge l_{k-1}-l_{\lceil\frac{k-1}{2}\rceil}$; thus, by swapping $P(T)$ with $Q(T)$, we can first assume that $l_{\lceil\frac{k-1}{2}\rceil} \ge \frac{l_{k-1}}{2}$, and then deduce the remaining cases using the relation between $P(T)$ and $Q(T)$.
		\item Assume that $\displaystyle P(T)=1+\sum_{i=l_1}^{deg P(T)} a_iT^i$ contains at least one term whose degree is not a multiple of $l_1$, and let $r$ be the smallest such degree. Then $P(T)^d$ has a term of degree $r$, that is, $r \in \{l_2,\dots,l_{k-1}\}$.
	\end{enumerate}
\end{remark}

The first step of our study consists in limiting the value of $d$ in function of the number of terms of the power $P(T)^d$. The next result can be deduced from known Theorems of Zannier and Schinzel (\cite{S2}, \cite{SZ1}); we include here an elementary proof.
\begin{proposition}\label{1}
	Let $d \ge 2$ and $k \ge 2$ be integers, and let $\xi_1,\dots,\xi_{k-1} \in \mathbb{C} \setminus \{0\}$ and $l_1 < l_2 < \dots < l_{k-1}$ be positive integers. Consider a polynomial $P(T) \in \mathbb{C}[T]$ such that $\displaystyle P(T)^d=1+\sum_{i=1}^{k-1} \xi_i T^{l_i}$. Then $d \le k-1$, and moreover, if $d=k-1$, then $P(T)=1+\frac{\xi_1}{d}T^{l_1}$.
\end{proposition}

\begin{proof}
	
	Clearly, all root multiplicities of $P(T)^d$ are divisible by $d$. Take a root $\alpha$ of $P(T)^d$, and let $\lambda d$ be its multiplicity. Then, using the substitution $Q(T)=P(T)^d$, we have that all derivatives $Q(\alpha)=\dv[j]{Q}{T}  (\alpha)$ vanish for every $j=1,\dots,\lambda d-1$.
	On the other hand $$\dv[j]{Q}{T} = \sum_{i=1}^{k-1} \xi_i l_i (l_i-1)\dots(l_i-j+1)T^{l_i-j}.$$
	Thus, the condition $\dv[j]{Q}{T}  (\alpha)=0$ for every $j=1,\dots,\lambda d-1$ can be naturally translated in a system consisting of $\lambda d-1$ equations. By multiplying the $i$th equation by $T^i$ (remember that $P(0) \neq 0$) we obtain the equivalent system $$\begin{cases} \displaystyle\sum_{i=1}^{k-1} \xi_i l_i T^{l_i}=0, \\ \displaystyle\sum_{i=1}^{k-1} \xi_i l_i (l_i-1)T^{l_i}=0, \\ \dots \\ \displaystyle\sum_{i=1}^{k-1} \xi_i l_i (l_i-1)\dots(l_i-\lambda d+2)T^{l_i}=0 . \end{cases}.$$
	At this stage, a solution of this system induces a solution of the associated linear system over $\mathbb{C}^k$ in the indeterminates $T^{l_1},\dots,T^{l_{k-1}}$, and the matrix associated to this linear system can be easily reduced to a Vandermonde matrix, whose determinant is non-vanishing: then, if $\lambda d > k-1$, this linear system admits only the trivial solution $T^{l_1}=\dots=T^{l_{k-1}}=0$, which is not admissible since $P(0) \neq 0$; therefore $d \le k-1$.
	
	Furthermore, arguing in the same way, if $d=k-1$ we have $\lambda=1$, hence all roots of $Q(T)$ have multiplicity $d$ and all roots of $P(T)$ are simple. Moreover, every solution of the linear system obtained from the vanishing condition on derivatives has the form $\mu(\alpha^{l_1},\dots,\alpha^{l_{k-1}})$, with $\mu \in \mathbb{C}$. Consider a root $\beta$ of $P(T)$ distinct from $\alpha$, there exists $\mu \in \mathbb{C}$ such that $(\beta^{l_1},\dots,\beta^{l_{k-1}})=\mu(\alpha^{l_1},\dots,\alpha^{l_{k-1}})$. Therefore, since $Q(\alpha)=0$, clearly $\displaystyle 1=-\sum_{i=1}^{k-1} \xi_i \alpha^{l_i}$, while on the other hand $Q(\beta)=0$ yields $\displaystyle 0=1+\mu \left(\sum_{i=1}^{k-1} \xi_i \alpha^{l_i} \right)=1-\mu$, that is, $\mu=1$. Then $\beta^{l_i}=\alpha^{l_i}$ for every $i=1,\dots,k-1$, implying that there exists a $m$th root of unity $\zeta$ such that $\beta=\zeta \alpha$, where $m$ is such that $m|l_i$ for every $i=1,\dots,k-1$. Then we obtain that $P(T)$ has at most $l_1$ roots. Thus, since those roots are all simple, it easily follows that $m=l_1=\deg P(T)$. The remaining claims follow immediately.\qedhere
	
\end{proof}

Another tool we need is a generalization of the \emph{Vandermonde's identity}. Remember that, for a real number $r \in \mathbb{R}$ that is not a negative integer, and for $n \in \mathbb{N}$, the \emph{binomial coefficient} $\binom{r}{n}$ is defined as $\binom{r}{n}:=\frac{r(r-1)(r-2)\dots (r-n+1)}{n!}$, and is such that the Binomial Theorem expansion $\displaystyle (1+x)^r=1+\sum_{i=1}^{+\infty}\binom{r}{i}x^i$ holds. Then, by expanding the right side of the identity $(1+x)=(1+x)^{\frac{1}{d}}\dots(1+x)^{\frac{1}{d}}$,  and noticing that all terms of degree greater than 1 of the expansion must vanish, we obtain the following lemma.
\begin{lemma}\label{3}
	Let $d$ and $n$ be integers greater than $1$. Then $$\sum_{x_1+\dots+x_d=n} \binom{1/d}{x_1}\binom{1/d}{x_2}\dots\binom{1/d}{x_d}=0.$$
\end{lemma}
Leveraging these preliminary results, we can prove the main result of this Section, that is, a characterization for complex polynomials $P(T) \in \mathbb{C}[T]$ having at most five terms.
\begin{proposition}\label{4}
	Let $d \ge 2$ and $k \le 5$ be positive integers, and let $\xi_1,\dots,\xi_{k-1} \in \mathbb{C} \setminus \{0\}$ and $l_1 < l_2 < \dots < l_{k-1}$ be positive integers. Assume that $P(T) \in \mathbb{C}[T]$ is a complex polynomial such that $\displaystyle P(T)^d=1+\sum_{i=1}^{k-1} \xi_i T^{l_i}$. Then the following tables describe the admissible values for coefficients and exponents of $P(T)$ and $P(T)^d$.
	
	\begin{enumerate}
		\item If $k=5$:
		\begin{center}
			\scalebox{0.8}{
				\begin{tabular}{ | c | c | c | c | c | c | c | c|}
					
					\hline
					$d$ & $l_2$ & $l_3$ & $l_4$ & $\xi_2$ & $\xi_3$ & $\xi_4$ & $P(T)$ \\ \hline
					\rule[-2.5ex]{0pt}{6ex} $4$ & $2l_1$ & $3l_1$ & $4l_1$ & $\frac{3}{8}\xi_1^2$ & $\frac{1}{16}\xi_1^3$ & $\frac{1}{256}\xi_1^3$& $1+\frac{1}{4}\xi_1T^{l_1}$\\ \hline
					\rule[-2.5ex]{0pt}{6ex} $3$ & $3l_1$ & $5l_1$ & $6l_1$ & $-\frac{5}{27}\xi_1^3$ & $\frac{1}{81}\xi_1^5$ & $-\frac{1}{729}\xi_1^6$& $1+\frac{1}{3}\xi_1T^{l_1}-\frac{1}{9}\xi_1^2T^{2l_1}$\\ \hline
					\rule[-2.5ex]{0pt}{6ex} $2$ & $4l_1$ & $5l_1$ & $6l_1$ & $\frac{5}{64}\xi_1^4$ & $-\frac{1}{64}\xi_1^5$ & $\frac{1}{256}\xi_1^6$& $1+\frac{1}{2}\xi_1T^{l_1}-\frac{1}{8}\xi_1^2T^{2l_1}+\frac{1}{16}\xi_1^3T^{3l_1}$\\ \hline
					\rule[-2.5ex]{0pt}{6ex} $2$ & $3l_1$ & $5l_1$ & $6l_1$ & $-\frac{5}{32}\xi_1^4$ & $\frac{1}{256}\xi_1^5$ & $\frac{19}{1024}\xi_1^6$& $1+\frac{1}{2}\xi_1T^{l_1}-\frac{1}{8}\xi_1^2T^{2l_1}-\frac{1}{64}\xi_1^3T^{3l_1}$\\ \hline
					\rule[-2.5ex]{0pt}{6ex} $2$ & $4l_1$ & $7l_1$ & $8l_1$ & $\frac{7}{64}\xi_1^4$ & $-\frac{1}{512}\xi_1^7$ & $\frac{1}{4096}\xi_1^8$& $1+\frac{1}{2}\xi_1T^{l_1}-\frac{1}{8}\xi_1^2T^{2l_1}+\frac{1}{16}\xi_1^3T^{3l_1}+\frac{1}{64}\xi_1^4T^{4l_1}$\\ \hline
					\rule[-2.5ex]{0pt}{6ex} $2$ & $2l_1$ & $5l_1$ & $6l_1$ & $\frac{5}{4}\xi_1^2$ & $-\frac{1}{4}\xi_1^5$ & $\frac{1}{16}\xi_1^6$& $1+\frac{1}{2}\xi_1T^{l_1}+\frac{1}{2}\xi_1^2T^{2l_1}-\frac{1}{4}\xi_1^3T^{3l_1}$\\ \hline
			\end{tabular}}
			\captionof{table}{}\label{Tab4-1}
		\end{center}
		
		\begin{center}
			\scalebox{0.8}{
				\begin{tabular}{ | c | c | c | c | c | c | c | c|}
					\hline
					$d$ & $l_2$ & $l_3$ & $l_4$ & $\xi_2$ & $\xi_3$ & $\xi_4$ & $P(T)$ \\ \hline
					\rule[-2.5ex]{0pt}{6ex} $2$ & $2l_1$ & $3l_1$ & $4l_1$ &  & $-\frac{1}{8}\xi_1^3+\frac{1}{2}\xi_1\xi_2$ & $\left(-\frac{1}{8}\xi_1^2+\frac{1}{2}\xi_2\right)^2$& $1+\frac{1}{2}\xi_1T^{l_1}+\left(\frac{1}{2}\xi_2-\frac{1}{8}\xi_1^2\right)T^{2l_1}$\\ \hline
			\end{tabular}}
			\captionof{table}{}\label{Tab4-2}
		\end{center}
		\item If $k=4$:
		\begin{center}
			\begin{tabular}{ | c | c | c | c | c | c |}
				\hline
				$d$ & $l_2$ & $l_3$ & $\xi_2$ & $\xi_3$ & $P(T)$ \\ \hline
				\rule[-2.5ex]{0pt}{6ex}
				$2$ & $3l_1$ & $4l_1$ & $-\frac{1}{8}\xi_1^3$ &  $-\frac{1}{64}\xi_1^4$ & $1+\frac{1}{2}\xi_1T^{l_1}-\frac{1}{8}\xi_1^2T^{2l_1}$\\ \hline
				$3$ & $2l_1$ & $3l_1$ & $\frac{1}{3}\xi_1^2$ &  $\frac{1}{27}\xi_1^3$ & $1+\frac{1}{3}\xi_1T^{l_1}$\\ \hline
			\end{tabular}
			\captionof{table}{}\label{Tab3}
			
		\end{center}
		\item If $k=3$:
		\begin{center}\scalebox{0.8}{
				\begin{tabular}{ | c | c | c | c |}
					\hline
					$d$ & $l_2$ & $\xi_2$ & $P(T)$ \\ \hline
					\rule[-2.5ex]{0pt}{6ex}
					$2$ & $2l_1$ & $\frac{1}{4}\xi_1^2$ &  $1+\frac{1}{2}\xi_1T^{l_1}$\\ \hline
			\end{tabular}}
			\captionof{table}{}\label{Tab2}
			
		\end{center}
	\end{enumerate}
\end{proposition}

\begin{proof}
	
	The last two items are a reformulation of (\cite[Lemma 2.1]{CZ1}). 
	
	Assume $k=5$. By Proposition \ref{1} we have $d \le 4$. Hence we have to study the three cases $d = 2,3,4$: in each case, we will focus on determining $P(T)$, and then the parameters can be deduced from the expansion of $P(T)^d$. Moreover, thanks to Remark \ref{2}, we can assume that $l_2 \ge \frac{l_4}{2}$ first, and then deduce the remaining solutions.
	\\
	\fbox{$d=4$}. By Proposition \ref{1} we immediately obtain $P(T)=1+\frac{1}{4}\xi_1T^{l_1}$.
	\\
	\fbox{$d=3$}. Since $\deg P(T)	=\frac{l_4}{3} < \frac{l_4}{2} \le l_2$, from the second part of Remark \ref{2} it follows that there exist $m \ge 1$, positive integers $1=\sigma_1 < \sigma_2 < \dots < \sigma_m$ and $\alpha_1,\dots,\alpha_m \in \mathbb{C}\setminus \{0\}$ such that $\displaystyle P(T)=1+\sum_{i=1}^{m}\alpha_iT^{\sigma_il_1}$. Taking $\rho(T)=\xi_1+\xi_2T^{l_2-l_1}+\xi_3T^{l_3-l_1}+\xi_4T^{l_4-l_1}$ and comparing the terms of degree $\sigma_il_1$ of $P(T)$ with those of the binomial expansion $$(1+T^{l_1}\rho(T))^{\frac{1}{3}}=1+\sum_{i=1}^{+\infty} \binom{1/3}{i} T^{il_1}\rho(T)^i,$$ we have $\alpha_i=\binom{1/3}{\sigma_i}\xi_1^{\sigma_i}$. 
	
	Let $r$ be the smallest positive integer such that $P(T)$ has no term of degree $rl_1$ (clearly $rl_1 \le \frac{l_4}{3}+l_1$). From the minimality of $r$ we can deduce that the coefficient of degree $rl_1$ of $P(T)^3$ is equal to $\displaystyle\left[ \left( \sum_{p,q,s \in \mathbb{N}}^{p+q+s=r} \binom{1/3}{p}\binom{1/3}{q}\binom{1/3}{s}\right)-3\binom{1/3}{r} \right] \xi_1^{r}$, which is non-zero by Lemma \ref{3}. Therefore $rl_1 \ge l_2 \ge \frac{l_4}{2}$, which implies $\frac{l_4}{2} \le \frac{l_4}{3}+l_1$ and $l_4 \le 6l_1$. Hence $\deg P(T) \le 2l_1$, and since $P(T)$ must contain at least three different terms (since otherwise $P(T)^3$ would have at most four terms), then $P(T)=1+\frac{1}{3}\xi_1T^{l_1}-\frac{1}{9}\xi_1^2T^{2l_1}$.
	\\
	\fbox{$d=2$}. We have $\deg P(T)=\frac{l_4}{2} \le l_2$; let us distinguish two cases.
	\begin{itemize}
		\item \fbox{$l_2 > \deg P(T)$}. Arguing as in the case $d=3$, we can infer that $\displaystyle P(T)=1+\sum_{i=1}^m \alpha_i T^{\sigma_il_1}$, with $\alpha_i=\binom{1/2}{\sigma_i}\xi_1^{\sigma_i}$; actually $l_1|\gcd(l_2,l_3,l_4)$ and $l_1|\frac{l_4}{2}$, and moreover, since $P(T)$ has at least three terms whose degree is lower than $l_2$, we must have $l_2 \ge 3l_1$ and $l_4 \ge 6l_1$.
		
		Now, let $r$ be the smallest positive integer such that $P(T)$ has no term of degree $rl_1$. Then, as before, we easily obtain that $rl_1 \ge l_2 > \frac{l_4}{2}$. Therefore $m=\frac{l_4}{2l_1} \ge 3$ and $\displaystyle P(T)=1+\sum_{i=1}^m \alpha_iT^{il_i}$, with $\alpha_i=\binom{1/2}{i}\xi_1^i$.
		
		In this setting, $P(T)^2$ has terms of degree $0,l_1,l_4,l_4-l_1$. Moreover, the term of degree $\frac{l_4}{2}+l_1$ of $P(T)^2$ is equal to $$\displaystyle\left[ \left( \sum_{i \in \mathbb{N}}^{m+1} \binom{1/2}{i}\binom{1/2}{m+1-i}\right)-2\binom{1/2}{m+1} \right] \xi_1^{m+1},$$ which is non-zero by Lemma \ref{3}, and the term of degree $l_4-2l_1$ of $P(T)^2$ is $$\displaystyle\left(2\binom{1/2}{m}\binom{1/2}{m-2}+\binom{1/2}{m-1}^2\right)\xi_1^{2m-2},$$ which is again non-zero because $\binom{1/2}{m}$ and $\binom{1/2}{m-2}$ have the same sign (notice that $m \ge 3$).
		
		Hence, two integers belonging to the set  $\{0,l_1,\frac{l_4}{2}+l_1,l_4-2l_1,l_4-l_1,l_4\}$ must coincide. It is simple to show that the only possible equality is $\frac{l_4}{2}+l_1=l_4-2l_1$, which yields $l_4=6l_1$, $m=3$ and $P(T)=1+\frac{1}{2}\xi_1T^{l_1}-\frac{1}{8}\xi_1^2T^{2l_1}+\frac{1}{16}\xi_1^3T^{3l_1}$.
		\item \fbox{$l_2 = \deg P(T)$}. Again, following the proof of the case $d=3$, we can deduce that $P(T)$ is a finite subsum of the infinite sum $\displaystyle 1+\left(\sum_{i=1}^{+\infty} \alpha_i T^{il_1}\right)+\frac{1}{2}\xi_2T^{l_2}$, with $\alpha_i=\binom{1/2}{i}\xi_1^i$.
		
		First, assume that $l_1$ is not a divisor of $l_2$. Then, if $\sigma_ml_1$ is the greatest degree multiple of $l_1$ among the terms of $P(T)$, $P(T)^2$ must have terms of degree $2l_2,l_2+\sigma_ml_1,2\sigma_ml_1,l_2,l_1,0$  (which are all pairwise distinct), contradicting our hypotheses. Then $l_1 | l_2$.
		
		Next, consider the smallest positive integer $r$ such that $P(T)$ has no term of degree $rl_1$: arguing as above, we can easily obtain that $rl_1 \ge l_2$. Moreover, since $\deg P(T) = l_2$, by studying the term of degree $l_2$ of the expansion of $P(T)^2$ we obtain $$\displaystyle P(T)= 1+\left(\sum_{i=1}^m \alpha_i T^{il_1}\right)+\frac{1}{2}\xi_2T^{l_2}, \ \ m=\frac{l_4}{2l_1}, \ \ \alpha_i=\binom{1/2}{i}\xi_1^i.$$
		Therefore, it is straightforward to see that $P(T)^2$ has non-zero terms of degrees $0,l_1,l_2,2l_2-l_1,2l_2$ (which are all pairwise distinct): thus, these must be the only terms of $P(T)^2$.
		
		At this point, we will reach a contradiction by considering the terms of degree $l_2+l_1$ and $2l_2-2l_1$ obtained expanding the square $P(T)^2$. In fact:
		\begin{itemize}
			\item The coefficient of the term of degree $l_2+l_1$ is equal to $-2\binom{1/2}{m+1}\xi_1^{m+1}+\frac{1}{2}\xi_1\xi_2$, according to Lemma \ref{3}.
			\item The coefficient of the term of degree $2l_2-2l_1$ is $$\left[2\binom{1/2}{m}\binom{1/2}{m-2}+\binom{1/2}{m-1}^2\right]\xi_1^{2m+2}+\binom{1/2}{m-2}\xi_1^{m-2}\xi_2.$$
		\end{itemize}
		Clearly, if one of those coefficients is non-zero, the associated degree must be among the ones listed before. Thus we have two possible cases:
		\begin{enumerate}
			\item If at least one of those coefficients is non-zero, then the associated degree must belong to $\{0,l_1,l_2,2l_2-l_1,2l_2\}$. It is trivial to check that if the associated degree is $l_1+l_2$, we must have $l_1+l_2=2l_2-l_1$, while if this degree is $2l_2-2l_1$ then we must have $2l_2-2l_1=l_2$; in both cases we deduce from the equality that $l_2=2l_1$, that is, $P(T)=1+\frac{1}{2}\xi_1T^{l_1}+\left(\frac{1}{2}\xi_2-\frac{1}{8}\xi_1^2\right)T^{2l_1}.$
			\item If both coefficients are non-zero, we can solve for $\xi_2$ in one of the two equations, and then, substituting in the other one, we obtain $\displaystyle \xi_2=4\binom{1/2}{m+1}\xi_1^{ml_1}$ and $$4\binom{1/2}{m+1}=-2\binom{1/2}{m}-\binom{1/2}{m-1}\frac{\frac{1}{2}-m+2}{m-1},$$
			or equivalently, that $$\frac{4\left(\frac{1}{2}-m+1\right)\left(\frac{1}{2}-m\right)}{m(m+1)}=-\frac{2\left(\frac{1}{2}-m+1\right)}{m}-\frac{\frac{1}{2}-m+2}{m-1}.$$
			This is a cubic equation in $m$, whose roots are $\frac{1}{2},3,4$. Then, in this case we have the two solutions:
			\begin{itemize}
				\item $m=3$, $\xi_2=4 \binom{1/2}{4}\xi_1^{3l_1}$ and  $P(T)=1+\frac{1}{2}\xi_1T^{l_1}-\frac{1}{8}\xi_1^2T^{2l_1}-\frac{1}{64}\xi_1^3T^{3l_1}$.
				\item $m=4$, $\xi_2=4 \binom{1/2}{5}\xi_1^{4l_1}$  $P(T)=1+\frac{1}{2}\xi_1T^{l_1}-\frac{1}{8}\xi_1^2T^{2l_1}+\frac{1}{16}\xi_1^3T^{3l_1}+\frac{1}{64}\xi_1^4T^{4l_1}$.
			\end{itemize}
		\end{enumerate}
	\end{itemize}
	The previous cases describe all polynomials $P(T)$ whose power has five terms, and such that $l_2 \ge \frac{l_4}{2}$. Then, we only have to study the case $l_2 < \frac{l_4}{2}$. For that purpose, thanks to Remark \ref{2} we know that there is a polynomial $Q(T)$, associated to $P(T)$ and such that $Q(T)^d=1+\xi_1'T^{l_1'}+\xi_2'T^{l_2'}+\xi_3'T^{l_3'}+\xi_4'T^{l_4'}$, with $l_2' > \frac{l_4'}{2}$. Therefore, $Q(T)$ must be one of the solutions obtained in the previous cases; however, there is only one solution satisfying $l_2' > \frac{l_4'}{2}$, that is,  $Q(T)=1+\frac{1}{2}\xi_1'T^{l_1'}-\frac{1}{8}\xi_1'^2T^{2l_1'}+\frac{1}{16}\xi_1'^3T^{3l_1'}$, which yields  $P(T)=1+\frac{1}{2}\xi_1T^{l_1}+\frac{1}{2}\xi_1^2T^{2l_1}-\frac{1}{4}\xi_1^3T^{3l_1}$, thus concluding our proof.
\end{proof}

\section{Closed forms of linear recursions and Universal Hilbert Sets}
Let $A \subseteq \mathbb{C}$ be a ring. Denote by $\mathcal{E}_A$ the ring of complex functions defined over $\mathbb{N}$ of the form $\displaystyle \alpha(n)=\sum_{i=1}^k c_i \alpha_i^n$, with $k \ge 2$, $c_1,\dots,c_k \in \mathbb{Q}$ and $\alpha_1,\dots,\alpha_k \in A$. In this work, $A$ will usually be either $\mathbb{Z}$ or $\mathbb{Q}$; in these cases, we further denote by $\mathcal{E}_A^+$ the subring formed by functions having only positive roots $\alpha_i$. These functions are closed forms for linear recurrence relations of order $k$ having only simple roots, and the behaviour of these functions under composition and exponentiation is the main subject of several papers (see \cite{CZ5, FK, Y}). Our interest in this topic stems from the characterization presented in \cite{CZ5}, which in turn was motivated by a question posed by Yasumoto in \cite{Y}, asking whether the set $\{2^n+3^n\}$ is a \emph{Universal Hilbert Set}, that is, a set $H$ such that for every polynomial $P(X,Y) \in \mathbb{Q}[X,Y]$ irreducible over $\mathbb{Q}$, the specialized polynomial $P(h,Y) \in \mathbb{Q}[Y]$ is irreducible for every $h \in H$, except at most for a finite set of values.

Our work is based on the following characterization proved by Corvaja and Zannier.
\begin{theorem}[\protect{\cite[Theorem 4]{CZ5}}]\label{charHS}
	For $\alpha \in \mathcal{E}_{\mathbb{Z}}^+$, the following conditions are equivalent:
	\begin{enumerate}[(i)]
		\item $\alpha(\mathbb{N})$ is a Universal Hilbert Set;
		\item there exist no integer $d \ge 2$, a polynomial $P(X) \in \mathbb{Q}[X]$ of degree $d$ and an element $\beta \in \mathcal{E}_\mathbb{Z}$ such that $\alpha'=P(\beta)$, where $\alpha'(n)=\alpha(dn)$.
	\end{enumerate}
\end{theorem}

This result allowed the authors to prove a generalization of Yasumoto's question. 
\begin{corollary}[\protect{\cite[Corollary 3]{CZ5}}]\label{indmul}
	Let $\displaystyle \alpha(n)=\sum_{i=1}^k c_i\alpha_i^n \in \mathcal{E}_{\mathbb{Z}}^+$ be such that $\alpha_1,\dots,\alpha_n \in \mathbb{Z}^+$ are multiplicatively independent. Then the set $\alpha(\mathbb{N})$ is a Universal Hilbert Set.
\end{corollary}

Write $\alpha(n)=\sum_{i=1}^k c_i\alpha_i^n \in \mathcal{E}_\mathbb{Z}^+$, and assume that $\alpha(\mathbb{N})$ is \textbf{not} a Universal Hilbert Set. Therefore, by the previous result there must exist a polynomial $P(T) \in \mathbb{Q}[T]$ of degree $d$ and $\beta(n)=\sum_{j=1}^h f_j\beta_j^n \in \mathcal{E}_\mathbb{Z}$ such that $\alpha'=P(\beta)$, where $\alpha'(n)=\alpha(dn)$, thus obtaining an identity of the form
\begin{equation}\label{charP1}
P(\beta)=\sum_{i=1}^k c_i\alpha_i^{dn}=\sum_{i=1}^k c_i (\alpha_i^d)^n.
\end{equation}

We can choose among the integers $\beta_1,\ldots,\beta_h$ a subset of multiplicatively independent integers with maximal cardinality $\sigma \ge 1$, which we can rename for simplicity $\{\beta_1,\ldots,\beta_\sigma\}$. Then, the maximality condition guarantees that if we add another element $\beta_{\sigma+1}$ to this subset, the elements of the new set will not be multiplicatively independent, and must therefore satisfy a relation of the form $\displaystyle \prod_{i=1}^{\sigma+1} \beta_i^{m_i}=1$ for some suitable integers $m_1,\ldots,m_{\sigma+1} \in \mathbb{Z}$. Repeating these steps for all other elements $\beta_{\sigma+1},\ldots,\beta_h$ we obtain a set of equations of the form
\begin{equation}\label{charP2}
\beta_i^{m_{ii}}=\beta_1^{m_{i1}}\beta_2^{m_{i2}}\dots\beta_\sigma^{m_{i\sigma}}
\end{equation}
for $i=\sigma+1,\dots,h$, and $m_{i1},\dots,m_{i\sigma},m_{ii} \in \mathbb{Z}$.

On the other hand, from identity (\ref{charP1}), we obtain that each element of the form $\alpha_i^d$ must be a term of the expansion of $P(\beta)$, and can be thus expressed as a monomial in $\beta_1,\dots,\beta_h$, which, in light of (\ref{charP2}), is actually a monomial in $\beta_1,\dots,\beta_\sigma$, yielding something of the form
\begin{equation}\label{charP3}
\alpha_i^d=\beta_1^{v_{i1}}\beta_2^{v_{i2}}\dots\beta_\sigma^{v_{i\sigma}},
\end{equation}

for $i=1,\ldots,k,$ and $v_{i1},\dots,v_{i\sigma} \in \mathbb{Q}$.

In order to turn (\ref{charP1}) in a polynomial identity, we can send the term associated to each element of our subset $\{\beta_1,\dots,\beta_\sigma\}$ to a power $(\beta_j)^n \mapsto Y_j^{r_j}$. Hence, we can deduce from (\ref{charP2}) and (\ref{charP3}) the image $(\beta_i)^{m_{ii}} \mapsto T_i(Y_1,\dots,Y_\sigma)$, and that there exists $M \in \mathbb{Z}^+$ such that $(\alpha_i^d)^M$ has integer exponents in the terms of equation (\ref{charP3}), and hence $[(\alpha_i^d)^M]^n \mapsto R_i(Y_1,\dots,Y_\sigma)$ (notice that the exponents of the monomials $T_i$ and $R_i$ might be negative). Next, we choose integers $r_j$ such that we can pick monomials $\tilde{T}_i(Y_1,\dots,Y_\sigma)$ and $\tilde{R}_i(Y_1,\dots,Y_\sigma)$ satisfying $\tilde{T}_i(Y_1,\dots,Y_\sigma)^{m_{ii}}=T_i(Y_1,\dots,Y_\sigma)$ and $\tilde{R}_i(Y_1,\dots,Y_\sigma)^M=R_i(Y_1,\dots,Y_\sigma)$, thus obtaining $(\beta_i)^n \mapsto \tilde{T}_i(Y_1,\dots,Y_\sigma)$ and $(\alpha_i^d)^n \mapsto \tilde{R}_i(Y_1,\dots,Y_\sigma)$.

Using the described map, $\beta(n)$ becomes a certain polynomial $G(Y_1,\dots,Y_\sigma)$, while identity (\ref{charP1}) yields a polynomial composition $$P(G(Y_1,\dots,Y_\sigma))=\sum_{i=1}^k c_i\tilde{R}_i(Y_1,\dots,Y_\sigma).$$

Clearly, we can choose a set $\mathcal{M}$ of multiplicative independent monomials with maximal cardinality $\sigma$, and map these monomials to a set of indeterminates $Y_1,\ldots,Y_\sigma$ such that each element of $\mathcal{M}$ becomes a power $Y_i^{l_i}$, where the exponents $l_i$ are chosen so that all exponents of our composition are still in $\mathbb{Z}$. With this change of variables, we can easily see that this problem actually asks for solutions of our main question.

Let $k=\sigma+\rho$ be the number of terms of the composition $f(g(X_1,\dots,X_\sigma))$. In the context given by Question \ref{mainQ}, Corollary \ref{indmul} deals with the case $\rho=0$; thus, in order to extend their result, we have to study Question \ref{mainQ} for fixed (small) values of $\rho$.

Also, it is worth noticing that all invariants related to our study are depending only on the cardinality $\sigma$ of the set of multiplicatively independent elements chosen. Thus, in some steps, we can pick a suitable set (with the same cardinality) of elements without losing information. Therefore in our results we will only list a subset of solutions such that any other solution can be obtained with an appropriate change of variables (dictated by the set of multiplicatively independent elements used).

\section{Small values of $\rho$}
Here, we will investigate the equation 		
\begin{equation}\label{mainEQ}
f(g(\vect{X}))=a_1X_1^{l_1}+\ldots+a_\sigma X_\sigma^{l_\sigma}+a_{\sigma+1}T_1(\vect{X})+\ldots+a_{\sigma+\rho}T_\rho(\vect{X})  \tag{1}
\end{equation}
for small values of $\rho$.
\begin{remark}\label{rho0}
	The only solution of equation $(*)$ for $\rho=0$ is $\sigma=1$, $f(T)=T^s$, $g(X_1)=cX_1^a$, since a solution with $\rho=0$ and $\sigma > 1$ would produce a counterexample to Corollary \ref{indmul}.
\end{remark}


Take a polynomial composition $f(g(X_1,\dots,X_\sigma))$ with at least $\sigma+1$ terms (i.e. $\rho \ge 1$). This composition must include some monomials $T(X_1,\dots,X_\sigma)$. We focus on the case where there are few such monomials. 

Our approach is recursive. Starting from the equation 
\begin{equation}
f(g(\vect{X}))=a_1X_1^{l_1}+\dots+a_\sigma X_\sigma^{l_\sigma}+a_{\sigma+1}T_1(X_1,\dots,X_\sigma)+\dots+a_{\sigma+\rho}T_\rho(X_1,\dots,X_\sigma)\tag{*}
\end{equation}
for a fixed value of $\rho$, we look for a specialization for one variable, say, $X_\sigma$, as a function of the others, such that equation $(*)$ is reduced to the equation
\begin{equation}
f(\tilde{g}(\vect{X}))=a_1X_1^{l_1}+\dots+a_{\sigma'} X_{\sigma'}^{l_{\sigma'}}+a_{{\sigma'}+1}T_1(X_1,\dots,X_{\sigma'})+\dots+a_{{\sigma'}+{\rho'}}T_{\rho'}(X_1,\dots,X_{\sigma'})\tag{**}
\end{equation}
where both parameters $\sigma'$ and $\rho'$ are strictly lower than $\sigma$ and $\rho$. Thus, assuming that equation $(**)$ has already been solved for all values of $\rho'$ up to $\rho-1$, we can deduce from that solution the polynomials $f(T)$ and $\tilde{g}(X_1,\dots,X_{\sigma'})$, and deduce the value of $\sigma$ from $\sigma'$, $\rho'$ and our specialization. With this information we try to deduce the inner polynomial $g(X_1,\dots,X_\sigma)$, solving our equation. 

We use this approach to deal with the cases $\rho=1,2$.

\begin{proposition}\label{rho1}
	Let $\sigma \ge 1$ be an integer, $g(X_1,\dots,X_\sigma) \in \mathbb{C}[X_1^{\pm 1},\dots,X_\sigma^{\pm 1}]$ be a Laurent polynomial in the indeterminates $\vect{X} = (X_1,\dots,X_\sigma)$, and let $f(T) \in \mathbb{C}[T]$ be such that
	$$
	f(g(\vect{X}))=a_1X_1^{l_1}+\dots+a_\sigma X_\sigma^{l_\sigma}+a_{\sigma+1}T_1(\vect{X}),
	$$
	with $l_1,\dots,l_\sigma$ positive integers, $a_1,\dots,a_{\sigma+1} \in \mathbb{C}$ and $T_1 \in \mathbb{C}[X_1^{\pm 1},\dots,X_\sigma^{\pm 1}]$ monomial in $X_1,\dots,X_\sigma$. 
	
	Then, up to a suitable change of variables, we have one of the following:
	\begin{enumerate}
		\item $\sigma=1$, $f(T)=T^{m_1}+cT^{m_2}$, $g(X_1)=\sqrt[S]{a_1}X_1^r$, with $m_1r=l_1$, $m_2r=l_2$, $c=a_2a_1^{-\frac{m_2}{m_1}}$.
		\item $\sigma=2$, $f(T)=T^2$, $g(X_1,X_2)=\sqrt{a_1}X_1^{\frac{l_1}{2}}+\sqrt{a_2}X_2^{\frac{l_2}{2}}$.
	\end{enumerate}
\end{proposition}
\begin{proof}
	First, assume that $\sigma=1$, that is,  $f(g(X_1))=a_1X_1^{l_1}+a_2X_1^{l_2}$. Let $D_1,D_2, d_1$ be respectively, the maximum, second maximum and minimum degree of the polynomial $g$. Similarly, let $m_1$ and $m_2$ be, respectively, the maximum and minimum degree of $f$ (hence $D_1 \ge D_2 \ge d_1$ and $m_1 \ge m_2$). We can expand $f(g(X_1))$ as $$f(g(X_1))=\sum_{j \in I_f} f_j \left(\sum_{i \in I_g}g_i X_1^i \right)^j,$$
	and the expansion on the right has only one term for each of the degrees $m_1D_1$, $(m_1-1)D_1+D_2$ and $m_2d_1$ (or $m_1d_1$ if $d_1 < 0$); thus, at least two among those degrees must coincide. However, we can easily show that this can happen only if $D_1=D_2=d_1$, that is, if $g(X_1)$ is a monomial $bX_1^r$, and $f(T)=T^{m_1}+cT^{m_2}$. Thus, we obtain $b=\sqrt[s]{a_1}$ and $m_1r=l_1, m_2r=l_2$, $c_2b^s=a_2$.
	
	Assume then $\sigma > 1$, and write $T_1(X_1,\dots,X_\sigma)=X_1^{v_{1}}\dots X_{\sigma}^{v_{\sigma}}$. We look for a specialization, depending on the values of $v_i$:
	\begin{enumerate}
		\item If there is an index $i$ such that $l_i \neq v_i$, hence, assuming that this index is $\sigma$, we impose the identity $a_\sigma X_\sigma^{l_\sigma} = - a_{\sigma+1}T_1(X_1,\dots,X_{\sigma})$. Since $l_\sigma \neq v_\sigma$, this equation yields a specialization of the form $X_\sigma = \tilde{T}(X_1,\dots,X_{\sigma-1})$ (where we can assume that the exponents of $\tilde{T}$ are integers, up to changing the variables $X_i$ with some suitable roots). 
		\item If $l_i=v_i$ for every $i=1,\dots,\sigma$, then the change of variables $Y_i=X_i^{l_i}$ maps the equation in $$f(g(Y_1,\dots,Y_{\sigma}))=a_1Y_1+\dots+a_\sigma Y_{\sigma}+a_{\sigma+1}Y_1\dots Y_\sigma.$$
		In this case, we choose $Y_\sigma=0$.
	\end{enumerate} 
	
	In both cases, the specialization is such that the original equation is reduced to the identity
	$$f(\tilde{g}(X_1,\dots,X_{\sigma-1}))=a_1X_1^{l_1}+\dots+a_{\sigma-1} X_{\sigma-1}^{l_{\sigma-1}}.$$
	Then, from Remark \ref{rho0} it follows that  $\sigma-1=1$, $f(T)=T^s$ and $\tilde{g}(X_1)=bX_1^r$. Hence, $$f(g(X_1,X_2))=g(X_1,X_2)^s=a_1X_1^{l_1}+a_2X_2^{l_2}+a_3 T_1(X_1,X_2).$$
	
	Define a function $\varphi: \mathbb{C}[X_1,X_2] \rightarrow
	\mathbb{C}[T]$, depending on $T_1(X_1,X_2)$, defined by the images of the two indeterminates $\varphi(X_1)=T^{m_1}$ and $\varphi(X_2)=T^{m_2}$, such that all terms involved in the expansion of $g(X_1,X_2)^s$ are mapped to different terms. Hence, with this mapping, the previous equation becomes 
	$$P(T)^s=a_1T^{\omega_1}+a_2T^{\omega_2}+a_3T^{\omega_3}.$$
	
	It has been proved (see Proposition \ref{4}, or \cite[Lemma 2.1]{CZ1}) that the only solution of this equation is $s=2$ and $P(T)$ is a binomial. Therefore, $g(X_1,X_2)$ is also a binomial, where the two terms are multiplicatively independent, and there is no cancellation in the expansion of $g(X_1,X_2)^2$. Hence, it suffice to apply a suitable change of variables such that those two multiplicatively independent terms become powers of the two indeterminates $X_1$ and $X_2$ (with a slight notation abuse) to conclude that $g(X_1,X_2)=\sqrt{a_1}X_1^{\frac{l_1}{2}}+\sqrt{a_2}X_2^{\frac{l_2}{2}}$.
\end{proof}
\begin{proposition}\label{rho2}
	Let $\sigma \ge 2$ be an integer, $g(X_1,\dots,X_\sigma) \in \mathbb{C}[X_1^{\pm 1},\dots,X_\sigma^{\pm 1}]$ a Laurent polynomial in the indeterminates $\vect{X} = (X_1,\dots,X_\sigma)$, and let $f(T) \in \mathbb{C}[T]$ be such that
	\begin{equation}
	f(g(\vect{X}))=a_1X_1^{l_1}+\dots+a_\sigma X_\sigma^{l_\sigma}+a_{\sigma+1}T_1(\vect{X})+a_{\sigma+2}T_2(\vect{X}) \tag{*},
	\end{equation}
	with $l_1,\dots,l_\sigma$ positive integers, $a_1,\dots,a_{\sigma+2} \in \mathbb{C}$ and $T_1, T_2 \in \mathbb{C}[X_1^{\pm 1},\dots,X_\sigma^{\pm 1}]$ distinct monomials in $X_1,\dots,X_\sigma$. 
	
	Then, up to a suitable change of variables, we have one of the following:
	\begin{enumerate}
		\item $\sigma=2$,  $g(X_1,X_2)=\sqrt[3]{a_1}X_1^{\frac{l_1}{3}}+\sqrt[3]{a_2}X_2^{\frac{l_2}{3}}$, $f(T)=T^3$.
		\item $\sigma=2$,  $g(X_1,X_2)=\sqrt{a_1}X_1^{\frac{l_1}{2}}+\sqrt{a_2}X_2^{\frac{l_2}{2}}+i\sqrt[4]{4a_1a_2}X_1^{\frac{l_1}{4}}X_2^{\frac{l_2}{4}}$, $f(T)=T^2$.
	\end{enumerate}
\end{proposition}

\begin{proof}
	First, we can assume, up to a rearrangement, that the indeterminate $X_\sigma$ appears in $T_1$ and $T_2$ with different exponents (not necessarily non-zero), that is, there exist integers $m_{\sigma 1} \neq m_{\sigma 2}$  such that $T_1=X_\sigma^{m_{\sigma 1}}\tilde{T}_1$ and $T_2=X_\sigma^{m_{\sigma 2}}\tilde{T}_2$, with $\tilde{T}_1,\tilde{T}_2$ monomials not containing $X_\sigma$. Hence, if we impose that $a_{\sigma+1}X_\sigma^{m_{\sigma 1}}\tilde{T}_1=-a_{\sigma+2}X_\sigma^{m_{\sigma 2}}\tilde{T}_2$ and solve for $X_\sigma$, we obtain a specialization $X_\sigma = \tilde{a} \tilde{T}$, where $\tilde{T}$ is a monomial in $X_1,\dots,X_{\sigma-1}$ with rational exponents. However, we can assume without loss of generality (by changing $X_1,\dots,X_{\sigma-1}$ with some suitable roots) that these exponents are integers, i.e. $\tilde{T}$ is a Laurent monomial in $X_1,\dots,X_{\sigma-1}$. 
	
	With this specialization, we can turn equation $(*)$ in another equation
	
	\begin{equation}
	f(\tilde{g}(\vect{X}))=a_1X_1^{l_1}+\dots+a_{\sigma-1}X_{\sigma-1}^{l_{\sigma-1}} + a_\sigma \tilde{a}^{l_\sigma} \tilde{T}^{l_\sigma} \tag{**}
	\end{equation}
	in at most $\sigma-1$ indeterminates, where $\tilde{g}(X_1,\dots,X_{\sigma-1})=g(X_1,\dots,X_{\sigma-1},\tilde{T})$, and such that $\rho' \le 1$ (notice that $\tilde{T}$ might coincide with one of the monomials $X_i^{l_i}$). We are now in the hypotheses of Remark \ref{rho0} and Proposition \ref{rho1}, hence we have the following cases:
	\begin{enumerate}
		\item If $\tilde{T}^{l_\sigma}=X_i^{l_i}$ for some $i$ and $a_{\sigma}\tilde{a}^{l_\sigma}=-a_i$, then these two terms cancel out. Thus we fall in the case considered in Remark \ref{rho0}, yielding $\sigma-2=1$, $f(T)=T^s$. Then, equation $(*)$ becomes
		\begin{equation}
		g(\vect{X})^s=a_1X_1^{l_1}+a_2X_2^{l_2}+a_3X_3^{l_3}+a_4T_1(X_1,X_2,X_3) + a_5T_2(X_1,X_2,X_3)\tag{E1}.
		\end{equation}
		\item If $\tilde{T}^{l_\sigma}=X_i^{l_i}$ for some $i$ and $a_{\sigma}\tilde{a}^{l_\sigma} \neq -a_i$, then these two terms can be merged in a single non-zero term. Then, again by Remark \ref{rho0} we deduce that $\sigma-1=1$ and $f(T)=T^s$, thus yielding
		\begin{equation}
		g(X_1,X_2)^s=a_1X_1^{l_1}+a_2X_2^{l_2}+a_3T_1(X_1,X_2) + a_4T_2(X_1,X_2)\tag{E2}.
		\end{equation}
		\item Lastly, if the exponents of $\tilde{T}^{l_\sigma}$ are different from those of the other terms of equation $(**)$, we can study equation $(**)$ by applying Proposition \ref{rho1}, obtaining two cases associated to the two solutions described therein.
		\begin{enumerate}
			\item $\sigma-1=1$, $f(T)=T^{s_1}+cT^{s_2}$, thus the equation becomes 
			\begin{equation}
			f(g(X_1,X_2))=a_1X_1^{l_1}+a_2X_2^{l_2}+a_3T_1(X_1,X_2) + a_4T_2(X_1,X_2)\tag{E3}.
			\end{equation}
			\item $\sigma-1=2$, $f(T)=T^2$, yielding again equation (E1), with $s=2$.
		\end{enumerate}
	\end{enumerate}
	
	Then, to conclude our study, we have to examine those three equations:
	\begin{enumerate}
		\item First, consider the equation
		\begin{equation}
		g(\vect{X})^s=a_1X_1^{l_1}+a_2X_2^{l_2}+a_3X_3^{l_3}+a_4T_1(X_1,X_2,X_3) + a_5T_2(X_1,X_2,X_3)\tag{E1}.
		\end{equation}
		
		With a suitable parametrization of the form $X_i \rightarrow T^{m_i}$ - such that all terms of equation (E1) have distinct images - we can turn our equation to the following polynomial identity in the single indeterminate $T$ 
		$$P(T)^s=a_1T^{\omega_1}+a_2T^{\omega_2}+a_3T^{\omega_3}+a_4T^{\omega_4}+a_5T^{\omega_5}.$$
		Clearly, we can assume without loss of generality that $\omega_1=0, a_1=1$, thus falling under the hypotheses of Proposition \ref{4}. Hence the solutions of this last equation are described in the following Tables:
		
		\begin{center}
			\scalebox{0.8}{
				\begin{tabular}{ | c | c | c | c | c | c | c | c|}
					
					\hline
					$s$ & $\omega_3$ & $\omega_4$ & $\omega_5$ & $a_3$ & $a_4$ & $a_5$ & $P(T)$ \\ \hline
					\rule[-2.5ex]{0pt}{6ex} $4$ & $2\omega_2$ & $3\omega_2$ & $4\omega_2$ & $\frac{3}{8}a_2^2$ & $\frac{1}{16}a_2^3$ & $\frac{1}{256}a_2^3$& $1+\frac{1}{4}a_2T^{\omega_2}$\\ \hline
					\rule[-2.5ex]{0pt}{6ex} $3$ & $3\omega_2$ & $5\omega_2$ & $6\omega_2$ & $-\frac{5}{27}a_2^3$ & $\frac{1}{81}a_2^5$ & $-\frac{1}{729}a_2^6$& $1+\frac{1}{3}a_2T^{\omega_2}-\frac{1}{9}a_2^2T^{2\omega_2}$\\ \hline
					\rule[-2.5ex]{0pt}{6ex} $2$ & $4\omega_2$ & $5\omega_2$ & $6\omega_2$ & $\frac{5}{64}a_2^4$ & $-\frac{1}{64}a_2^5$ & $\frac{1}{256}a_2^6$& $1+\frac{1}{2}a_2T^{\omega_2}-\frac{1}{8}a_2^2T^{2\omega_2}+\frac{1}{16}a_2^3T^{3\omega_2}$\\ \hline
					\rule[-2.5ex]{0pt}{6ex} $2$ & $3\omega_2$ & $5\omega_2$ & $6\omega_2$ & $-\frac{5}{32}a_2^4$ & $\frac{1}{256}a_2^5$ & $\frac{19}{1024}a_2^6$& $1+\frac{1}{2}a_2T^{\omega_2}-\frac{1}{8}a_2^2T^{2\omega_2}-\frac{1}{64}a_2^3T^{3\omega_2}$\\ \hline
					\rule[-2.5ex]{0pt}{6ex} $2$ & $4\omega_2$ & $7\omega_2$ & $8\omega_2$ & $\frac{7}{64}a_2^4$ & $-\frac{1}{512}a_2^7$ & $\frac{1}{4096}a_2^8$& $1+\frac{1}{2}a_2T^{\omega_2}-\frac{1}{8}a_2^2T^{2\omega_2}+\frac{1}{16}a_2^3T^{3\omega_2}+\frac{1}{64}a_2^4T^{4\omega_2}$\\ \hline
					\rule[-2.5ex]{0pt}{6ex} $2$ & $2\omega_2$ & $5\omega_2$ & $6\omega_2$ & $\frac{5}{4}a_2^2$ & $-\frac{1}{4}a_2^5$ & $\frac{1}{16}a_2^6$& $1+\frac{1}{2}a_2T^{\omega_2}+\frac{1}{2}a_2^2T^{2\omega_2}-\frac{1}{4}a_2^3T^{3\omega_2}$\\ \hline
			\end{tabular}}
			\captionof{table}{}\label{Tab4.1}
		\end{center}
		
		\begin{center}
			\scalebox{0.8}{
				\begin{tabular}{ | c | c | c | c | c | c | c | c|}
					\hline
					$d$ & $\omega_3$ & $\omega_4$ & $\omega_5$ & $a_3$ & $a_4$ & $a_5$ & $P(T)$ \\ \hline
					\rule[-2.5ex]{0pt}{6ex} $2$ & $2\omega_2$ & $3\omega_2$ & $4\omega_2$ &  & $-\frac{1}{8}a_2^3+\frac{1}{2}a_2a_3$ & $\left(-\frac{1}{8}a_2^2+\frac{1}{2}a_3\right)^2$& $1+\frac{1}{2}a_2T^{\omega_2}+\left(\frac{1}{2}a_3-\frac{1}{8}a_2^2\right)T^{2\omega_2}$\\ \hline
			\end{tabular}}
			\captionof{table}{}\label{Tab4.2}
		\end{center}
		
		Now, notice that since $g(X_1,X_2,X_3)^s$ contains the $3$ multiplicatively independent terms $X_i^{l_i}$, then $g(X_1,X_2,X_3)$ also must contain at least three multiplicatively independent terms. Moreover, if $g(X_1,X_2,X_3)$ has exactly three terms, these must be all multiplicatively independent, thus there are no cancellations in the expansion of $g(X_1,X_2,X_3)^s$, which then will have at least $\binom{4}{2}=6$ terms. Hence, looking at the Tables we can deduce that $s=2$, and $g(X_1,X_2,X_3)$ has $4$ or $5$ terms. Since in both cases there must be exactly three multiplicatively independent terms, we can map those three terms each in a power of one indeterminate, thus obtaining an equation of the form $$g(X_1,X_2,X_3)=c_1X_1^{\alpha_1}+c_2X_2^{\alpha_2}+c_3X_3^{\alpha_3}+c_4R_1(X_1,X_2,X_3)+c_5R_2(X_1,X_2,X_3)$$
		(where $g(X_1,X_2,X_3)$ has four terms if and only if $R_1=R_2$). Then, expanding $g(X_1,X_2,X_3)^2$ it is easy to check that among the terms dependent on $R_1,R_2$ we cannot impose equalities such that there are only five terms left in the end.
		\item Consider now the equation
		\begin{equation}
		g(X_1,X_2)^s=a_1X_1^{l_1}+a_2X_2^{l_2} + a_3 T_1(X_1,X_2) + a_4 T_2(X_1,X_2) \tag{E2}.
		\end{equation}
		Similarly, using a suitable parametrization, depending on $T_1$ and $T_2$, of the form $X_1 \rightarrow T^{m_1}$, $X_2 \rightarrow T^{m_2}$ such that no terms of (E2) have the same image, we reduce our equation to
		$$P(T)^s=a_1T^{\omega_1}+a_2T^{\omega_2}+a_3T^{\omega_3}+a_4T^{\omega_4}.$$
		Again, we can assume without loss of generality that $\omega_1=0, a_1=1$. Thus, as a consequence of Proposition \ref{4}, we have the following solutions:
		
		\begin{center}
			\begin{tabular}{ | c | c | c | c |}
				\hline
				$s$ & $\omega_3$ & $\omega_4$ & $P(T)$ \\ \hline
				\rule[-2.5ex]{0pt}{6ex}
				$2$ & $3\omega_2$ & $4\omega_2$ & $1+\frac{1}{2}a_2T^{\omega_2}-\frac{1}{8}a_2^2T^{2\omega_2}$\\ \hline
				\rule[-2.5ex]{0pt}{6ex}
				$3$ & $2\omega_2$ & $3\omega_2$ &  $1+\frac{1}{3}a_2T^{\omega_2}$\\ \hline
			\end{tabular}
		\end{center}
		Essentially, these two solutions can be described as the cube of a binomial, and the square of a special trinomial. 
		Next, we examine in detail these two cases:
		\begin{itemize}
			\item If $s=3$, and $g(X_1,X_2)$ is a binomial, since the two terms of $g(X_1,X_2)$ must be multiplicatively independent (because $g(X_1,X_2)^3$ has two multiplicatively independent terms), we can assume without loss of generality (up to a suitable change of variables) that $g(X_1,X_2)=c_1X_1^{\alpha_1}+c_2X_2^{\alpha_2}$, which immediately yields the first solution.
			\item If $s=2$, and $g(X_1,X_2)$ is a trinomial, as before, there must be at least two multiplicatively independent terms in $g(X_1,X_2)$. On the other hand, clearly the three terms cannot be multiplicatively independent (since in this case the square would have six terms). Moreover, we deduce from the table that in $P(T)^2$ the square of the middle term cancels out with the mixed product of the other two. Thus, reflecting this in our original equation, we can assume, up to a suitable change of variables, that $g(X_1,X_2)$ has the form $$g(X_1,X_2)=c_1X_1^{\alpha_1}+c_2X_2^{\alpha_2}+c_3R(X_1,X_2),$$ with $R(X_1,X_2)$ distinct from the other two terms, and such that one of the following conditions hold (up to a rearrangement):
			\begin{enumerate}
				\item $2c_2c_3X_2^{\alpha_2}R(X_1,X_2)=-c_1^2X_1^{2\alpha_1}$, that is, $R(X_1,X_2)=X_1^{2\alpha_1}X_2^{-\alpha_2}$, $c_3=-\frac{c_1^2}{2c_2}$.
				\item $c_3^2R^2(X_1,X_2)=-2c_1c_2X_1^{\alpha_1}X_2^{\alpha_2}$, which implies $R(X_1,X_2)=X_1^{\frac{\alpha_1}{2}}X_2^{\frac{\alpha_2}{2}}$, $c_3=\sqrt{-2c_1c_2}$.
			\end{enumerate}
			However, it is easy to see that these two solutions are equivalent up to a change of variables, obtained by mapping $R(X_1,X_2)$ to one indeterminate (or to an appropriate root). Therefore, expanding the square yields the second solution.
		\end{itemize}
		\item Finally, consider now the equation \begin{equation}
		f(g(X_1,X_2))=a_1X_1^{l_1}+a_2X_2^{l_2}+a_3T_1(X_1,X_2) + a_4T_2(X_1,X_2)\tag{E3},
		\end{equation} 
		with $f(T)=T^{s_1}+\alpha T^{s_2}$, $s_1 > s_2 \ge 1$. In this context, we have $\tilde{g}(X_1)=cX_1^r$, where $\tilde{g}(X_1)=g(X_1,\tilde{T})$, and $X_2=\tilde{T}$ is obtained from $(*)$. In particular, $\tilde{T}$ is a monomial in the single indeterminate $X_1$, that is, $\tilde{T}=t_1X_1^{\gamma_1}$. But since $g(X_1,X_1^{\gamma_1})=cX_1^r$, the terms (disregarding the coefficients) of $g(X_1,X_2)$ must necessarily be contained in the sum $\displaystyle \sum_{i \in \mathbb{Z}} X_1^{r-i\gamma_1} X_2^i.$ At this point, write $$\displaystyle g(X_1,X_2)= \sum_{i \in I} X_1^{r-i\gamma_1} X_2^i,$$
		and consider $f(g(X_1,X_2))=g(X_1,X_2)^{s_1}+\alpha g(X_1,X_2)^{s_2}$. We want to prove that the two polynomials $g(X_1,X_2)^{s_1}$ and $g(X_1,X_2)^{s_2}$ have no common term: in fact, if such a monomial $X_1^{N_1}X_2^{N_2}$ exists, there must be $i_1,\dots,i_{s_1}$ and $j_1,\dots,j_{s_2}$ such that $$\begin{cases} (r-\gamma_1i_1)+\dots+(r-\gamma_1i_{s_1})=N_1, \\ i_1+\dots+i_{s_1} = N_2, \\  (r-\gamma_1j_1)+\dots+(r-\gamma_1j_{s_2})=N_1,  \\ j_1+\dots+j_{s_2}=N_2,  \end{cases} \text{ implying } \begin{cases} i_1+\dots+i_{s_1} = N_2, \\ s_1r-\gamma_1N_2=N_1, \\ j_1+\dots+j_{s_2}=i_1+\dots+i_{s_1},\\
		s_2r-\gamma_1N_2=N_1.  \end{cases}.$$
		The second and the fourth equation of this linear system are incompatible if $s_1 \neq s_2$; hence, the two polynomials $g(X_1,X_2)^{s_1}$ and $g(X_1,X_2)^{s_2}$ have no common term. The number of terms of $f(g(X_1,X_2))$ is then equal to the sum of the number of terms of $g(X_1,X_2)^{s_1}$ and $g(X_1,X_2)^{s_2}$. Since $s_1 > s_2 \ge 1$, and $g(X_1,X_2)$ has at least two terms, $g(X_1,X_2)^{s_1}$ has at least three terms, while $g(X_1,X_2)^{s_2}$ has at least two: therefore $f(g(X_1,X_2))$ has at least $5 > 4 =\sigma+\rho$ terms, which is a  contradiction.\qedhere
	\end{enumerate}

\end{proof}

A direct application of this result yields an extension of Corollary \ref{indmul}.

\begin{corollary}\label{12dep}
	Let $\alpha(n)=\displaystyle \sum_{i=1}^k c_i\alpha_i^n \in \mathcal{E}_{\mathbb{Z}^+}$. 
	\begin{enumerate}
		\item If there are exactly $k-1 \ge 2$ (and not more) multiplicatively independent integers among the elements of the set $	\{\alpha_1,\dots,\alpha_k\}$, then the set $\alpha(\mathbb{N})$ is a Universal Hilbert Set, unless $\alpha(n)$ is of the form $\alpha(n)=(b_1\beta_1^m+b_2\beta_2^m)^2$.
		\item If there are exactly $k-2 \ge 2$ (and not more) multiplicatively independent integers among the elements of the set $	\{\alpha_1,\dots,\alpha_k\}$, then the set $\alpha(\mathbb{N})$ is a Universal Hilbert Set, unless $\alpha(n)$ is of the form $\alpha(n)=(b_1\beta_1^m+b_2\beta_2^m)^3$.
	\end{enumerate}
	
	In particular, if $\alpha(n)=\alpha_1^n+\dots+\alpha_k^n$, and among the $\alpha_i$ there are at least $k-2 \ge 2$ multiplicatively independent elements, then $\alpha(n)$ is a Universal Hilbert Set. 
\end{corollary}

\begin{proof}
	Let $\alpha(n)=\sum_{i=1}^k c_i\alpha_i^n$ be such that $\alpha(\mathbb{N})$ is not a Universal Hilbert Set. Then by Theorem \ref{charHS} there exist $\beta \in \mathcal{E}_\mathbb{Z}$ and a polynomial $P(T) \in \mathbb{Q}[T]$ of degree $d \ge 2$ such that $\alpha(dn)=P(\beta(n))$ identically.
	\begin{enumerate}
		\item We can map this relation to an equation of the form $$F(G(X_1,\dots,X_\sigma))=a_1X_1^{l_1}+\dots+a_\sigma X_\sigma^{l_\sigma}+a_{\sigma+1}T_1(X_1,\dots,X_\sigma).$$
		Hence, since $k \ge 3$, by Proposition \ref{rho1} we obtain $F(T)=T^2$ and that $G(X_1,X_2)$ is a binomial.
		\item In this case, the relation $\alpha(dn)=P(\beta(n))$ is mapped, via a suitable change of variables, to the equation
		$$F(G(X_1,\dots,X_\sigma))=a_1X_1^{l_1}+\dots+a_\sigma X_\sigma^{l_\sigma}+a_{\sigma+1}T_1(X_1,\dots,X_\sigma)+a_{\sigma+2}T_2(X_1,\dots,X_\sigma).$$
		Therefore, since $k \ge 3$, we are under the assumptions of Proposition \ref{rho2}. However, since both polynomials $F(T)$ and $G(X_1,\ldots,X_\sigma)$ cannot contain terms with non-real cofficients, the only admissible solution is the one where $F(T)=T^3$ and $G(X_1,X_2)$ is a binomial. \qedhere
	\end{enumerate}
\end{proof}

\begin{example}
	\begin{enumerate}
		\item From the previous proposition we immediately deduce that any set of the form $\{\alpha_1^n+\alpha_2^n+\alpha_3^n\}$, or $\{\alpha_1^n+\alpha_2^n+\alpha_3^n+\alpha_4^n\}$, where the integers $\alpha_i$ are not powers of the same integer, is a Universal Hilbert Set.
		\item Let $\alpha(n)=8^n+27^n+3\cdot 12^n+3 \cdot 18^n$. Since $\alpha(n)=(2^n+3^n)^3$, $\alpha(\mathbb{N})$ is not a Universal Hilbert Set.\qedhere
	\end{enumerate}
\end{example}

\section{The general case}
One of the drawbacks of the recursive strategy described i is that such a method produces mechanical proofs with an increasing number of cases that would not give much insight on the general case, thus making this kind of result pointless and tedious. However, it makes sense to ask how many multiplicatively independent elements there can be in a polynomial composition $f(g(X_1,\dots,X_\sigma))$. Namely, we will ask the following: \emph{if we fix the number of variables $\sigma$, can we determine the minimum number $k$ of terms of a composition?}

Let $\sigma \ge 1$, $\rho \ge 0$ be integers, $f \in \mathbb{C}[T]$ a polynomial of degree larger than $1$, $g \in \mathbb{C}[X_1^{\pm 1},\dots,X_\sigma^{\pm 1}]$ a Laurent polynomial in $X_1,\dots,X_\sigma$, and consider the equation of Question \ref{mainQ} $$f(g(X_1,\dots,X_\sigma))=a_1X_1^{l_1}+\dots+a_\sigma X_{\sigma}^{l_\sigma}+a_{\sigma+1}T_1(X_1,\dots,X_\sigma)+\dots+a_{\sigma+\rho}T_\rho(X_1,\dots,X_\sigma),$$
where $a_1,\dots,a_{\sigma+\rho} \in \mathbb{C}$, $l_1,\dots,l_\sigma \in \mathbb{Z}$, and $T_1,\dots,T_\rho$ are Laurent monomials in $X_1,\dots,X_\sigma$. Fix $\sigma + \rho = k$, and write the two polynomials $f$ and $g$ as  $$f(T)=\sum_{j \in J} f_j T^j \ \ \ \text{  and  } \ \ \ g(X_1,\dots,X_\sigma)=\sum_{(i_1,\dots,i_\sigma) \in I} g_{i_1,\dots,i_\sigma}X_1^{i_1}\dots X_{\sigma}^{i_\sigma},$$
where both polynomials have coefficients in $\mathbb{C}$, and the sets $J \subset \mathbb{Z}$ and $I \subset \mathbb{Z}^{\sigma}$ are finite, with $h=|I|$ being the number of terms of the inner polynomial $g(X_1,\dots,X_\sigma)$. We want to provide a lower bound for the minimum number of terms $k=\sigma + \rho$ of these polynomial compositions, depending on $\sigma$ and $h$; we will denote this minimum by $\tilde{k}(\sigma,h)$. Moreover, denote by $\displaystyle \tilde{k}(\sigma)=\min_{k \ge \sigma} \tilde{k}(\sigma,h)$.

Clearly, considering $g(X_1,\dots,X_\sigma)=X_1+\dots+X_\sigma$ and $f(T)=T^2$ we see that $f(g(X_1,\dots,X_\sigma))$ has exactly $\binom{\sigma+1}{2}$ terms, thus $\tilde{k}(\sigma) \le \binom{\sigma+1}{2}$.

On the other hand, we can easily prove that, if $k \le 2\sigma - 2$, there are no polynomials $f$ and $g$ such that
\begin{equation}\label{mainEQ}
f(g(\vect{X}))=a_1X_1^{l_1}+\ldots+a_\sigma X_\sigma^{l_\sigma}+a_{\sigma+1}T_1(\vect{X})+\ldots+a_{\sigma+\rho}T_\rho(\vect{X})  \tag{1}
\end{equation}

For this purpose, consider the two monomials $T_1$ and $T_2$. Since they are distinct, we can deduce from the assumption $a_{\sigma+1}T_1(X_1,\dots,X_\sigma)=-a_{\sigma+2}T_2(X_1,\dots,X_\sigma)$ a specialization $X_i = \tilde{T}$ for one of our indeterminates. Moreover, if we apply this specialization, our equation will be reduced to another one having $\sigma' \le \sigma-1$ variables and $k' \le k-2$ terms; further, in order to cancel out one variable $X_j$, the term $a_jX_j^{l_j}$ must cancel out with another term, which has to be associated (after our specialization) to a certain monomial - that is, the only way to cancel out a variable is to diminish the number of terms (at least) by two. Therefore, it is easy to check that this second equation would still satisfy $k' \le 2\sigma' -2$; hence we can prove by a descent argument (notice that we have no solution for $\sigma = 1,2$) that $\tilde{k}(\sigma) \ge 2 \sigma-1$.

It is worth noticing that this argument can be applied to Theorem \ref{charHS}, yielding the following result.

\begin{proposition}\label{trivbnd}
	Let $\displaystyle \alpha(n)=\sum_{i=1}^k c_i\alpha_i^n \in \mathcal{E}_{\mathbb{Z}^+}$ be such that there are at least $\frac{k}{2}+1$ multiplicatively independent elements between $\alpha_1,\dots,\alpha_n \in \mathbb{Z}^+$. 
	
	Then the set $\alpha(\mathbb{N})$ is a Universal Hilbert Set.
\end{proposition}

\subsection{Small values of $\sigma$}
We have proved that $$2\sigma-1 \le \tilde{k}(\sigma) \le \binom{\sigma+1}{2}.$$
In order to guess whether those two bounds are sharp and gather more insight on $\tilde{k}(\sigma)$, we study the behavior of $\tilde{k}(\sigma)$ for small values of $\sigma$.
\begin{remark}\label{k123}
	\begin{enumerate}
		\item Clearly, $\tilde{k}(1)=1=2 \cdot 1 -1=\binom{1+1}{2}$ (see Remark \ref{rho0}).
		\item Remark \ref{rho0} and Proposition \ref{rho1} yield $\tilde{k}(2)=3=\binom{2+1}{2}$.
		\item Since $6=\binom{\sigma+1}{2}$, $\tilde{k}(3) \le 6$. On the other hand, for a composition having $k < 6$ terms, we are under the hypotheses of either Remark \ref{rho0} or Propositions \ref{rho1} and \ref{rho2}. Since these results show that there are no solutions for $\sigma=3$ under these hypotheses, $\tilde{k}(3) = 6 = \binom{3+1}{2}$.
	\end{enumerate}
\end{remark}
The case $\sigma=4$ is significantly harder to deal with our tools, since Proposition \ref{rho2} would not cover the cases $7 \le k \le 9$, and we already noted that, as $k$ grows, our recursive strategy becomes way more impractical. However, we can still show that $\tilde{k}(4)=\binom{4+1}{2}$ if $g(X_1,\dots,X_4)$ is a polynomial in $\mathbb{C}[X_1,X_2,X_3,X_4]$.
\begin{proposition}\label{k4}
	Let $\rho \ge 0$ be an integer, $g(X_1,X_2,X_3,X_4) \in \mathbb{C}[X_1,X_2,X_3,X_4]$ a polynomial in the indeterminates $\vect{X} = (X_1,X_2,X_3,X_4)$, and let $f(T) \in \mathbb{C}[T]$ be such that
	\begin{equation}
	f(g(\vect{X}))=a_1X_1^{l_1}+\dots+a_4 X_4^{l_4}+ \sum_{i=1}^{\rho} a_{4+i}T_i(\vect{X}) \tag{*},
	\end{equation}
	with $l_1,l_2,l_3,l_4$ positive integers, $a_1,\dots,a_{4+\rho} \in \mathbb{C}$ and $T_1, \dots,T_\rho \in \mathbb{C}[X_1,X_2,X_3,X_4]$ monomials in $X_1,\dots,X_4$. 
	Let $k=4+\rho$ be the number of terms of $f(g(X_1,X_2,X_3,X_4))$. 
	
	Then $k \ge 10=\binom{4+1}{2}$.
\end{proposition} 

\begin{proof}
We only need to work on the cases $k=7,8,9$.
	\begin{enumerate}
		\item Assume $k=7$. Assume that $X_4$ appears in the monomial $T_1$; therefore, with the specialization $X_4=0$, we obtain that $f(g(X_1,X_2,X_3,0))$ is still a composition of the form $(*)$ (with $3$ variables) with at most $5$ terms, which yields a contradiction since $\tilde{k}(3)=6$ by Remark \ref{k123}.
		\item Suppose $k=8$. Now, assume that there is one variable $X_i$ dividing at least two monomials $T_j, T_r$. Then, with the specialization $X_i=0$, we obtain again a composition of the form $(*)$ in $3$ variables having between three and five terms, yielding again a contradiction by Remark \ref{k123}. Hence, each variable divides at most one of the four monomials $T_i$: clearly, this leads to the equation (up to a rearrangement) 
		$$f(g(X_1,\dots,X_4))=a_1X_1^{l_1}+a_2X_2^{l_2}+a_3X_3^{l_3}+a_4 X_4^{l_4}+a_5X_1^{l'_1}+a_6X_2^{l'_2}+a_7X_3^{l'_3}+a_8 X_4^{l'_4}.$$
		However, at this point, if $\tilde{g}(X_1,X_2)=g(X_1,X_2,0,0)$, with the substitution $X_3=X_4=0$ we obtain $$f(\tilde{g}(X_1,X_2))=a_1X_1^{l_1}+a_2X_2^{l_2}+a_5X_1^{l'_1}+a_6X_2^{l'_2},$$ which by Proposition \ref{rho2} implies either $f(T)=T^2$ or $f(T)=T^3$.
		
		On the other hand, if $\bar{g}(X_1)=g(X_1,0,0,0)$, with the specialization $X_2=X_3=X_4=0$ we obtain $$f(\bar{g}(X_1))=a_1X_1^{l_1}+a_5X_1^{l'_1},$$ which in turn implies, by Proposition \ref{rho1}, that $f(T)$ has the form $f(T)=T^{m_1}+cT^{m_2}$, which is impossible.
		
		\item Finally, let $k=9$. In that case, clearly there exists one variable dividing at least two between the monomials $T_1,\dots,T_5$. Notice that if there is an indeterminate $X_i$ dividing at least three monomials, we would obtain once again a contradiction with the specialization $X_1=0$ and Remark \ref{k123}. Hence, each indeterminate appears in at most two monomials, and there is at least one variable appearing in exactly two. Let $X_1$ be one of the variables appearing in the maximum number of terms of our composition; then $X_1$ divides exactly two monomials (say, $T_4$ and $T_5$), and all other variables divide at most two monomials between $T_1,\dots,T_5$. Then, with the specialization $X_1=0$ we have $$f(g(0,X_2,X_3,X_4))=a_2X_2^{l_2}+a_3X_3^{l_3}+a_4X_4^{l_4}+ \sum_{i=1}^{3} a_{4+i}T_i(0,X_2,X_3,X_4).$$
		
		Now, let $X_2$ be the variable that divides the most monomials among $T_1,T_2,T_3$. As we said, $X_2$ must divide at most two of these three. We have two possible cases:
		\begin{enumerate}
			\item Assume that $X_2$ divides exactly one monomial (say, $T_3$). Then, setting $X_2=0$ we are left with the equation $$f(g(0,0,X_3,X_4))=a_3X_3^{l_3}+a_4X_4^{l_4}+  a_{5}T_1(0,0,X_3,X_4)+a_6T_2(0,0,X_3,X_4),$$
			with the added information that $X_3$ and $X_4$ both divide exactly one of the monomials $T_1,T_2$.
			This equation is a polynomial composition in $2$ variables; thus by Proposition \ref{rho2} we obtain that either $f(T)=T^2$ or $f(T)=T^3$. However, assuming that $X_3$ divides $T_2$, further setting $X_3=0$ we get 
			$$f(g(0,0,0,X_4))=a_4X_4^{l_4}+  a_{5}T_1(0,0,X_3,X_4),$$
			which implies by Proposition \ref{rho1} that $f(T)$ is of the form $f(T)=T^{m_1}+cT^{m_2}$, contradicting the previous statement.
			\item Suppose then that $X_2$ divides exactly two of these monomials (say, $T_2$ and $T_3$). Then by imposing $X_2=0$ we obtain $$f(g(0,0,X_3,X_4))=a_3X_3^{l_3}+a_4X_4^{l_4}+  a_{5}T_1(0,0,X_3,X_4),$$
			which by Proposition \ref{rho1} implies $f(T)=T^2$.
			
			Now, notice that, since there are four variables, each dividing at most two monomials, there must be a monomial $T_i$ containing exactly one variable, say $T_i=X_j^{l'_j}$. However, if there is exactly one monomial containing only the variable $X_j$ between $T_1,\dots,T_5$, clearly by sending all other variables to zero we would obtain an equation of the form $$f(\tilde{g}(X_j))=a_jX_j^{l_j}+a_{i+4}\tilde{T}_i(X_j),$$ which would imply by Proposition \ref{rho1} that $f(T)=T^{m_1}+cT^{m_2}$, contradicting the previous part; thus, since each indeterminate can divide at most two monomials, we can conclude (since each variable divides at most two monomials) that there is a variable $X_s$, dividing exactly two monomials, such that those two monomials do not contain other variables besides $X_s$. Now, let us rearrange our indexes such that this variable $X_s$ is $X_1$ and the two monomials $T_4$ and $T_5$ are of the form $T_4=X_1^{v_4}$, $T_5=X_1^{v_5}$ (we can do that since this does not contradict our previous assumption on $X_1$). 
			
			Therefore our main equation becomes $$g(X_1,X_2,X_3,X_4)^2=a_1X_1^{l_1}+a_2X_2^{l_2}+a_3X_3^{l_3}+a_4X_4^{l_4}+$$ $$+a_8X_1^{v_4}+a_9X_1^{v_5}+\sum_{i=1}^{3} a_{i+4} T_i(X_2,X_3,X_4).$$ 
			
			We will conclude the proof by showing that it is not possible for the square of a polynomial $g(X_1,\dots,X_\sigma)^2$ (assuming that this square has the form $(*)$) to contain no mixed product between $X_1$ and the other variables. 
			
			In fact, remembering that such a polynomial $g(X_1,\dots,X_\sigma)$ contains monomials consisting of single variables for each variable $X_i$, then it also contains some mixed products between $X_1$ and some other variables (else it would be impossible to cancel out the mixed products arising from the square expansion of $g(X_1,\dots,X_\sigma)$). Between those mixed products, pick the maximum one with respect to the natural lexicographic order, and denote it by $\tilde{m}(X_1,\dots,X_\sigma)$; also, denote by $\tilde{M}(X_1,\dots,X_\sigma)$ the maximum term of $g(X_1,\dots,X_\sigma)$. Clearly, the term $\tilde{m}(X_1,\dots,X_\sigma)\tilde{M}(X_1,\dots,X_\sigma)$ is a mixed product containing $X_1$ and some other variables, which appears in the square expansion of $g(X_1,\dots,X_\sigma)^2$; in order to check that this term does not cancel out (and thus appears in $g(X_1,\dots,X_\sigma)^2$), we consider two cases:
			\begin{enumerate}
				\item If $\tilde{m}(X_1,\dots,X_\sigma)=\tilde{M}(X_1,\dots,X_\sigma)$, then $\tilde{m}(X_1,\dots,X_\sigma)\tilde{M}(X_1,\dots,X_\sigma)=\tilde{M}^2(X_1,\dots,X_\sigma)$ appears only once in the expansion, and thus does not cancel out.
				\item If $\tilde{m}(X_1,\dots,X_\sigma)\neq \tilde{M}(X_1,\dots,X_\sigma)$, we have $\tilde{M}(X_1,\dots,X_\sigma)=X_1^{\alpha}$, and clearly it is not possible to realize  $\tilde{m}(X_1,\dots,X_\sigma)\tilde{M}(X_1,\dots,X_\sigma)$ in any other ways in the square expansion of $g(X_1,\dots,X_\sigma)$, and thus this term does not cancel out.
			\end{enumerate}
			
		\end{enumerate}
	\end{enumerate} 
\end{proof}
\subsection{Sum of sets of vectors and cancellations}
The previous results point to $\tilde{k}(\sigma)$ being close to $\binom{\sigma+1}{2}$; while a general result is out of reach at the moment, there is some evidence suggesting that this might indeed be the case. 
In fact, assuming fixed the number of terms $h=|I|$ of $g(X_1,\dots,X_\sigma)$, the number of terms $k$ of the composition $f(g(X_1,\dots,X_\sigma))$ is obtained by expanding the sum  \begin{equation}
f(g(X_1,\dots,X_\sigma))=\sum_{j \in J} f_j g(X_1,\dots,X_\sigma)^j \tag{$\star$},
\end{equation} and then cancelling out some terms.  Then, in order to study $\tilde{k}(\sigma,h)$, we can consider two invariants:
\begin{enumerate}
	\item The number of different exponents in the expansion $(\star)$ of $f(g(X_1,\dots,X_\sigma))$, counting all exponents appearing before any cancellation between terms belonging to different powers of $g(X_1,\dots,X_\sigma)$ is performed. We will denote this value by $W(f,g)$, and is obviously an upper bound for $k$.
	\item The number of exponents (among the $W(f,g)$ listed before) appearing in the expansion that are cancelled out in the final computation of $f(g(X_1,\dots,X_\sigma))$. We will denote this value by $C(f,g)$; obviously $k= W(f,g)-C(f,g)$.
\end{enumerate}

The motivation behind this division lies in the fact that the two invariants $W(f,g)$ and $C(f,g)$ describe very different problems.

In fact, we can study $W(f,g)$ using tools from additive number theory. By definition we can associate to the exponents of $g(X_1,\dots,X_\sigma)$ the set $I \subseteq \mathbb{Z}^\sigma$, and, for a fixed integer $\alpha \ge 1$, the exponents of $g(X_1,\dots,X_\sigma)^\alpha$ described by $W(f,g)$ are exactly the elements of the set $\alpha I$. Therefore, by studying $W(f,g)$ we are basically studying the cardinality of the union set $\displaystyle \bigcup_{\alpha \in J} \alpha I$, where $I \subseteq \mathbb{Z}^\sigma$ and $J \subseteq \mathbb{Z}$ are fixed. In this context, the following result, due to Ruzsa, comes handy.
\begin{theorem}[\protect{\cite[Corollary 1.1]{R2}}]
	Let $A,B \subseteq \mathbb{R}^\sigma$  be two sets such that $|A| \le |B|$, and assume that there is no proper hyperplane of $\mathbb{R}^\sigma$ containing the set $A+B$. Then $$|A+B| \ge |B|+\sigma |A|- \frac{\sigma(\sigma+1)}{2}.$$	
\end{theorem}

In our context, the set $I$ has dimension at least equal to $\sigma-1$ (since $f(g(X_1,\dots,X_\sigma))$ must contain $\sigma$ multiplicatively independent terms); then, if  $h=|I| \ge \sigma-1$, from the previous theorem we immediately deduce that $$|\alpha I | \ge h + (\alpha-1)\left[ (\sigma - 1) h - \frac{\sigma(\sigma-1)}{2}\right],$$
which yields $$W(f,g)=|\bigcup_{\alpha \in J} \alpha I | \ge |deg(f)I| \ge h + (deg(f)-1)\left[ (\sigma - 1)h - \frac{\sigma(\sigma-1)}{2}\right].$$ 

We can deduce from here that our claim is true in the naive case $C(f,g)=0$ (for instance if $f(T)$ and $g(X_1,\dots,X_\sigma)$ have positive real coefficients).

\begin{proposition}\label{sigmapos}
	Let $\sigma \ge 1$ and $\rho \ge 0$ be integers, $g\in\mathbb{C}[X_1^{\pm 1},\ldots,X_\sigma^{\pm 1}]$ be a Laurent polynomial, having exactly $h$ terms, in the indeterminates $\vect{X} = (X_1,\ldots,X_\sigma)$, and let $f \in \mathbb{C}[T]$ be such that $$f(g(\vect{X}))=a_1X_1^{l_1}+\dots+a_\sigma X_{\sigma}^{l_\sigma}+a_{\sigma+1}T_1(\vect{X})+\dots+a_{\sigma+\rho}T_\rho(\vect{X}),$$
	with $l_1,\dots,l_\sigma$ positive integers, $a_1,\dots,a_{\sigma+\rho} \in \mathbb{C}$ and $T_1, \dots, T_\rho \in \mathbb{C}[X_1^{\pm 1},\dots,X_\sigma^{\pm 1}]$ monomials in the indeterminates $X_1,\dots,X_\sigma$.
	Let $k = \sigma+\rho$ be the number of terms of this polynomial composition.
	
	Then
	$$W(f,g) \ge  h + (deg(f)-1)\left[ (\sigma - 1)h - \frac{\sigma(\sigma-1)}{2}\right] \ge \sigma h - \frac{\sigma(\sigma-1)}{2} \ge \binom{\sigma+1}{2}.$$
\end{proposition}
The lower bounds provided in Proposition \ref{sigmapos} are sharp; in fact, it is easy to see that, for $$f(T)=T^2 \ \ , \ \ \ g(X_1,\dots,X_\sigma)=X_1+\dots+X_\sigma+\sum_{i=2}^{h-\sigma+1} \frac{X_1^i}{X_\sigma^{i-1}},$$
the composition $f(g(X_1,\dots,X_\sigma))=g(X_1,\dots,X_\sigma)^2$ has exactly $\sigma h - \frac{\sigma(\sigma-1)}{2}$ terms.

Furthermore, we can rewrite the second bound as $$W(f,g) \ge \sigma h - \frac{\sigma(\sigma-1)}{2}=\binom{\sigma+1}{2}+\sigma(h-\sigma).$$
Therefore, if we could prove that $C(f,g) \le \sigma(h-\sigma)$, we would obtain the desired bound for $\tilde{k}(\sigma)$. Since $g(X_1,\dots,X_\sigma)$ must contain, by definition, at least $\sigma$ multiplicatively independent terms, this upper bound for $C(f,g)$ basically states that for each term of $g(X_1,\dots,X_\sigma)$, besides the $\sigma$ multiplicatively independent ones, there can be no more than $\sigma$ cancellations. 

However, studying $C(f,g)$ is an extremely hard task. In fact, each cancellation would imply a polynomial relation between the coefficients of the involved monomials (from the expansion $(\star)$) and an equality on the exponents.

The first problem can be reduced to a study of intersection of algebraic surfaces, which is in itself a very hard problem, given our little knowledge on the behaviour of these coefficients.

As for the second one, in order to have a cancellation involving terms of (at least) two different polynomial powers $g(X_1,\dots,X_\sigma)^{j_1}$ and $g(X_1,\dots,X_\sigma)^{j_2}$ there should exist a monomial $X_1^{\alpha_1}\dots X_\sigma^{\alpha_\sigma}$ appearing in both these powers; clearly, this is equivalent to saying that the vector $(\alpha_1,\dots,\alpha_\sigma) \in \mathbb{Z}^\sigma$ must belong to both sets $j_1I$ and $j_2I$. 

Thus we have to face the following additive problem:
\begin{question}\label{NS3factq}
	
	Let $\sigma \ge 1$, and consider two finite sets $I \subseteq \mathbb{Z}^\sigma$ and $J \subseteq \mathbb{Z}^+$. Set $|I|=h$. 
	
	Given a vector $\vect{w} \in \mathbb{Z}^\sigma$, determine all factorizations of $\vect{w}$ of the type $$\vect{w}=c_1\vect{v}_1+\dots+c_k\vect{v}_k,$$
	
	where $\vect{v}_1,\dots,\vect{v}_k \in I$ and $c_1,\dots,c_k \in \mathbb{Z}^+$ are such that $c_1+\dots+c_k \in J$.
\end{question}
Additive decompositions have been the subject of several works (\cite{NS3G} and \cite{S3G} are good monographs on this argument); in general, these decompositions are not unique, and finding the possible decompositions of a given vector with respect to a finite set of generators $I$ is very hard. In fact, the easiest case $\sigma=1$ (where our vectors are, actually, integers) is a reformulation of the well-known \emph{Subset Sum Problem} (\cite{GJSSP}), which asks, given a finite set $I \subseteq{\mathbb{Z}}$, if there exists a subset $J$ of $I$ such that the sum of the elements of $J$ is a target value $w$; however, the Subset Sum Problem is \textbf{NP}-complete (see \cite{GJSSP} for a proof).

In light of both the evidence provided and the final considerations, we conclude this work with the following question:
\begin{question}
	Is it true that $\tilde{k}(\sigma) = \binom{\sigma+1}{2}$?
\end{question}
	
\section*{Acknowledgements}
This work is part of my PhD thesis. I would like to thank my advisors, Professors Roberto Dvornicich and Umberto Zannier for their supervision, and for helpful discussions. I would also like to thank the referee for their helpful comments.

\appendix

\section{Perfect powers in base $x$ and polynomial powers}

In this Appendix, we will apply Proposition \ref{4} on lacunary polynomial powers to study perfect powers in a given base $x$ having exactly $k$ non-zero digits, using (a part of) a method developed by Corvaja and Zannier in \cite{CZ1}. First, notice that, dividing by a power of $x$, we can assume without loss of generality that the units digit is non-zero, obtaining the Diophantine equation
\begin{equation}\label{maincase}
y^d=c_0+\sum_{i=1}^{k-1} c_ix^{m_i},
\end{equation} 
with fixed $k \in \mathbb{N}$, $y,d,x$ positive integers greater than $1$, $c_0,c_1,\dots,c_{k-1} \in \{1,\dots,x-1\}$ and $m_1 < \dots < m_{k-1}$ positive integers. This problem is actually quite complex; to get a feel of its difficulty, notice that the case $k=2, c_0=c_1=1$ is the well-known Catalan Conjecture, which stood open for more than a century and was proved by Mihailescu in \cite{M1}.

Here, we describe Corvaja and Zannier's method. Recall that the logarithmic Weil height of a rational number $\frac{a}{b} \in \mathbb{Q}$ in lowest terms is defined as $h\left( \frac{a}{b}\right)=\log \max\{|a|,|b|\}$, with the assumption that $h(0)=0$. The absolute logarithmic Weil height of an element $\alpha$ in a number field $K$ is defined as $$h(\alpha)=\frac{1}{[K : \mathbb{Q}]}\sum_{v \in M_K} [K_v : \mathbb{Q}_v] \log \max\{|\alpha|_v,1\},$$
where $M_K$ is a normalized set of inequivalent absolute values $|\cdot|_v$ defined on $K$. Further, the Weil height of a point $P=[\alpha_0:\ldots:\alpha_n]$ in a projective space $\mathbb{P}_n(K)$ is $$h(P)=\frac{1}{[K : \mathbb{Q}]}\sum_{v \in M_K} [K_v : \mathbb{Q}_v] \log \max_{0 \le i \le n}\{|\alpha|_v\}.$$
Moreover, if $S \subset M_K$ is a finite set of absolute values of $K$ containing the archimedean ones $S_\infty$, we will say that $x \in K$ is a $S$-integer if $|x|_v \le 1$ for every $v \not \in S$, and we denote the ring of $S$-integer elements of $K$ by $\mathcal{O}_{S,K}$. Invertible elements in $\mathcal{O}_{S,K}$ are called $S$-units (see \cite{BG}, \cite{CZ3}), and several results in the literature deal with properties and distributions of $S$-integral points. In particular, our problem also falls in this category.

We will consider the following cases:
\\
\underline{\textbf{First case}}: Let $C_{k-2} \in ]0,1[$  be a constant, and assume that there is a linear gap between the two leftmost non-zero digits, that is,  $m_{k-1} \le C_{k-2}m_{k-2}$. 

Dividing equation (\ref{maincase}) by $x^{m_{k-1}}$, we obtain  $$y^dx^{-m_{k-1}}=c_0+\sum_{i=1}^{k-1} c_ix_i,$$
 where $n_i=m_{k-1}-m_{k-1-i}$, for $i=1,\dots,k-2$, $n_{k-1}=m_{k-1}$, and $x_i=x^{-n_i}$ for $i=1,\dots,k-1.$
Define now the series $F(X_1,\dots,X_{k-1}) \in \mathbb{Q}[\![X_1,\dots,X_{k-1}]\!]$ obtained applying the Binomial Theorem  $$F(X_1,\dots,X_{k-1}):=\left[c_0+\sum_{i=1}^{k-1} c_iX_i \right]^{\frac{1}{d}}=c_0^{\frac{1}{d}}\left[ 1 + \sum_{i=1}^{k-1}\frac{c_i}{c_0}X_i \right]^{\frac{1}{d}}  = $$ $$=c_0^{\frac{1}{d}} \left[ 1+\frac{1}{d}\left( \sum_{i=1}^{k-1} \frac{c_i}{c_0}X_i \right)+\frac{1-d}{2d^2}\left( \sum_{i=1}^{k-1} \frac{c_i}{c_0}X_i \right)^2+\dots \right].$$
This expression converges absolutely if $\displaystyle \sum_{i=1}^{k-1} \frac{c_i}{c_0}X_i < 1$, then, for instance, noticing that $c_i \le x-1$ for $i=0,\dots,k-1$, it converges absolutely for $X_1,\dots,X_{k-1} \in \mathbb{C}$ such that $|X_i| < \frac{1}{(k-1)(x-1)}$, to a function, which we will denote by (slightly abusing our notation) $F$, that takes the value $c_0^{\frac{1}{d}}$ at the origin and is such that $\displaystyle F(X_1,\dots,X_{k-1})^d=c_0+\sum_{i=1}^{k-1}c_iX_i$. Since $n_{k-1} > n_{k-2} > \dots > n_1 = m_{k-1}-m_{k-2} \ge (1-C_{k-2})m_{k-1}$, for sufficiently large values of $m_{k-1}$ we have $$x^{-n_{k-1}}=x^{-m_{k-1}} \le \frac{1}{(k-1)(x-1)},$$ thus the series converges at $(x_1, \dots,x_{k-1})$; moreover, taking $z=F(x_1,\dots,x_{k-1})$, this yields $$z^d=c_0+c_1x^{-n_1}+\dots+c_{k-1}x^{-n_{k-1}}.$$

Let $K$ be the splitting field of $Y^d-x$ over $\mathbb{Q}$; since $z=yx^{-\frac{m_{k-1}}{d}}$, our sequence of solutions is defined over $K$. Further, if $S$ is the finite set of  places defined over $K$ consisting of the ones lying over either $\infty$ or $x$; then, by definition, $z$ is an $S$-integer, while $x_1,\dots,x_{k-1}$ are $S$-units. Moreover, the usual absolute value on $\mathbb{C}$ induces an absolute value on $\mathbb{Q}(z)$, which we can further extend to an infinite place $v$ defined over $K$; thus, embedding $K$ in $\mathbb{C}$ by means of $v$ we obtain that $z=F(x_1,\dots,x_{k-1})$ also with respect to $v$-adic convergence.

It is easy to check that $\displaystyle \sum_{i=1}^{k-1} h(x_i) \le (k-1)m_{k-1}$; further, since we have $\max_i |x_i|_v \le x^{-(1-C_{k-2})m_{k-1}}$, it follows that $$\sum_{i=1}^{k-1} h(x_i)=O(-\log(\max_i |x_i|_v)),$$ whence the convergence implies that $h(z)=h(F(x_1,\dots,x_{k-1})) \le 2m_{k-1}$. Hence, we fall under the assumptions of the following theorem.
\begin{theorem}[\protect{\cite[Theorem 1]{CZ3}}]\label{fatone}
	Let $K$ be a number field, $v$ a place defined over $K$ and let $C_v$ a completion of an algebraic closure of $K_v$. Let $S$ be a finite set of absolute values of $K$ containing $S_\infty$, and define the \textbf{$S$-height}  $h_S(x)=\sum_{v \not \in S} \log^+ |x|_v$ of a non-zero element $x \in K^*$. 
	
	Next, let $f(X)=\displaystyle \sum_i a_i X^i$ be a power series with algebraic coefficients in $C_v$ converging in a neighborhood of the origin in $C_v^n$; let $x_h=(x_{h1},\dots,x_{hn})$, $h \in \mathbb{N}$ be a sequence in $(K^*)^n$ tending to the origin of $K_v^n$, such that $f(x_h)$ is well defined and belongs to $K$.
	
	Suppose that:
	\begin{enumerate}
		\item For $i=1,\dots,n$ we have $h_S(x_{hi})+h_S(x_{hi}^{-1})=o(h(x_{hi}))$ as $h \rightarrow +\infty$;
		\item $h(x_h)=O(-\log(\max_i |x_{hi}|_v))$;
		\item $h_S(f(x_h))=o(h(x_h))$;
		\item $h(f(x_h))=O(h(x_h))$.
	\end{enumerate}
	Then there exist a finite number of cosets $u_1H_1,\dots,u_rH_r \in \mathbb{G}_m^n$ such that $\{x_h\} \subset \bigcup_{i=1}^{r} u_iH_i$ and such that, for $i=1,\dots,r$, the restriction of $f(X)$ to $u_iH_i$ coincides with a polynomial in $K[X]$.
\end{theorem}

Thus there are a finite number of cosets $u_1H_1,\dots,u_rH_r \in \mathbb{G}_m^n$ such that our sequence of solutions belongs to the union of these cosets, and such that for $i=1,\dots,r$ the restriction of $F(X)$ to $u_iH_i$ coincides with a polynomial. Then, since the elements of our sequence are $S$-units, we can use the following known theorem.

\begin{theorem}[\protect{\cite[Theorem 7.4.7]{BG}}]\label{Zarclos}
	Let $\mathbb{Q}^\times$ be the multiplicative group of units of $\overline{\mathbb{Q}}$, and let $\Gamma$ be a finitely generated subgroup of $(\overline{\mathbb{Q}^\times})^n$; let $\Sigma$ be a subset of $\Gamma$. 
	
	Then the Zariski closure of $\Sigma$ in $\mathbb{G}_m^n$ is a finite union of translates of algebraic subgroups of $\mathbb{G}_m^n$.
\end{theorem}

In our setting, Theorem \ref{Zarclos} states that the Zariski closure of our sequence is a certain finite union of translates of algebraic subgroups of $\mathbb{G}_m^n$. Then, going to an appropriate infinite subsequence of solutions (and by taking intersection with one of our cosets), we can assume that there is a single coset $uH$ containing all our solutions, where $u=(\chi_1,\dots,\chi_{k-1}) \in \mathbb{G}_m^{k-1}$ is a solution of our sequence (thus $\chi_i \in x^{\mathbb{Z}}$, with negative exponent), and that our sequence is Zariski-dense in said coset. Clearly, this coset cannot be a single point, therefore $s := \dim H > 0$; since our sequence converges $v$-adically to the origin in $uH$, the following proposition delivers the promised relation between this problem and lacunary polynomial powers.

\begin{proposition}[\protect{\cite[Proposition 1]{CZ3}}]\label{eqv}
	Let $H$ be a connected algebraic subgroup of $\mathbb{G}_m^n$. Then the following conditions are equivalent.
	\begin{enumerate}
		\item The Zariski closure of $H$ in $\mathbb{A}^n$ contains $(0,\dots,0)$.
		\item The lattice $\Lambda_H$ does not contain any non-zero vectors with all non-negative coordinates.
		\item There exists a parametrization $\varphi: \mathbb{G}_m^k \rightarrow H$, with $k= \dim H$, given by $X_i=T^{u_i}, i=1,\dots,n, T=(T_1,\dots,T_k)$, where all coordinates of $u_i \in \mathbb{Z}^k$ are strictly positive.
		\item There exists a point $(x_1,\dots,x_n) \in H \cap C_v^n$, where $C_v^n$ is as in Theorem \ref{fatone}, such that $|x_i|_v < 1$ for $i=1,\dots,n$.
		\item There exists a sequence in $H \cap C_v^n$ converging to $(0,\dots,0)$ in the $v$-adic topology.
	\end{enumerate}
\end{proposition}

Our sequence satisfies the fifth condition of Proposition \ref{eqv}; then, every condition holds in our coset $uH$. In particular, the third one states that there exists a parametrization of $uH$ of the form $X_i=\tilde{\xi}_iT_1^{a_{i1}}\dots T_s^{a_{is}},$ with $i=1,\dots,k-1$ and $a_{ij} \ge 0$ for every $i,j$. Moreover, Theorem \ref{fatone} ensures that $F(X_1,\dots,X_{k-1})$ becomes, with this parametrization of $uH$, a polynomial in the indeterminates $T_1,\dots,T_s$. Since $F^d(x_1,\dots,x_{k-1})=c_0+c_1x_1+\dots+c_{k-1}x_{k-1}$, and our sequence of solutions is Zariski-dense in $uH$, we obtain a polynomial identity of the shape
$$F^d(X_1,\dots,X_{k-1})=c_0+\sum_{i=1}^{k-1 }c_i\tilde{\xi}_iT_1^{a_{i1}}\dots T_s^{a_{is}}.$$
Since our sequence converges $v$-adically to $(0,\dots,0)$, the vectors $\vect{a}_i=(a_{i1},\dots,a_{is}) \in \mathbb{N}^s$ are non-zero, hence there is a vector $\vect{b}=(b_1,\dots,b_s), b_i \in \mathbb{Z}^+$ such that the scalar products $l_i:=\vect{b}\vect{a}_i$ are all positive and $l_i=l_j$ if and only if $\vect{a}_i=\vect{a}_j$. Thus, replacing $T_j$ with $T^{b_j}$, for $j=1,\dots,s$, $F(X_1,\dots,X_{k-1})$ becomes a non-constant polynomial $P(T) \in \mathbb{C}[T]$ such that the following identity holds (dividing by $c_0$ if needed): 
\begin{equation}\label{mainpol}
P(T)^d=1+\sum_{i=1}^{k-1}\xi_iT^{l_i}.
\end{equation} 
At this stage, notice that, following these substitutions, our coefficients $\xi_i$ are such that  $\xi_i \in \mathcal{C}_x=\{\frac{px^q}{r} \ | \ p,r \in \{ 1,\dots,x-1\}, q \in \mathbb{Z}^- \}$, and that $l_1 \le l_2 \le l_3 \le l_4$ are positive integers, associated to the $n_i$ (not necessarily in the same order), such that $l_i=l_j$ implies $\xi_i \neq \xi_j$. Therefore, we can use classification results for lacunary polynomial powers with complex coefficients to solve equation (\ref{mainpol}), and then check for every solution whether the coefficients $\xi_i$ belong to $\mathcal{C}_x$; after that, we can then pull back some useful information that will allow us to solve the original equation (\ref{maincase}). A similar relation can also be obtained if the extremal gap involves the rightmost digits rather than the leftmost ones. 
\\
\underline{\textbf{Second case}}: Let $C_1 \in ]0,1[$ be a constant, and assume that there is a linear gap between the two rightmost non-zero digits, that is, $m_1 \ge C_1m_{k-1}$ (namely, the second leftmost non-zero digit grows linearly with the length of our perfect power). 

Let $x_i=x^{m_i}$; then $$y^d=c_0+\sum_{i=1}^{k-1} c_ix_i.$$
Then, in a similar fashion as in the first case, we can reduce, with some analogous parametrization, this equation to the same polynomial identity (\ref{mainpol}), with $\xi_i \in \mathcal{C}_x.$

We now study equation (\ref{mainpol}). Proposition \ref{4} allows us to solve this polynomial equation under the assumption $k \le 5$; clearly, since the cases $k=3$ and $k=4$ have already been investigated by Corvaja and Zannier, we will focus on the next case $k=5$; further, we will assume for simplicity that all the non-zero digits of the perfect powers in equation (\ref{maincase}) are equal to $1$, that is, $c_0=\dots=c_{k-1}=1$, although this method can be extended to deal with the other cases with some tedious calculations. Then $\mathcal{C}_x=x^{\mathbb{Z}^-}$, and all coefficients $\xi_i$ must be perfect powers of $x$ with negative exponent.

Therefore, we have to examine the tables described in Proposition \ref{4}; our aim is to rewrite the polynomials contained therein as polynomials with at most five terms, whose	 degrees are not necessarily different, and with coefficients belonging to $x^{\mathbb{Z}^-}$:
\begin{itemize}
	\item Table \ref{Tab4-1}: Here $P(T)^d$ has exactly five terms, hence we must have $l_1 < l_2 < l_3 < l_4$ and $\xi_i \in x^{\mathbb{Z}}$ for every $i=1,\dots,4$.  Since the coefficients $\xi_i$ must be positive, only the first solution is admissible; however, since $\xi_i \in x^\mathbb{Z}$, for that case we would obtain $\frac{3}{8}=\frac{\xi_2}{\xi_1^2} \in x^{\mathbb{Z}}$, which is impossible.
	\item Table \ref{Tab4-2}: In this case, all coefficients are depending on  $\xi_1,\xi_2$. Thus, taking $\xi_1=x^a$, $\xi_2=x^b$, the condition $\xi_3 \in x^{\mathbb{Z}}$ implies that $\xi_3=\xi_1\left( \frac{4\xi_2-\xi_1^2}{8} \right) \in x^\mathbb{Z}$, that is $\frac{4x^b-x^{2a}}{8} = x^c$ for some $c \in \mathbb{Z}$ (such that $\xi_3=x^{c+a}$), or equivalently $\frac{4x^{b-2a}-1}{8}=x^{c-2a}$. However, it is easy to see that if $x \neq 2$ the left-hand side of the equation cannot belong to  $x^{\mathbb{Z}}$; on the other hand, if $x=2$, the equation becomes $2^{b+2}-2^{2a}=2^{c+3}$, whose solutions are $c+3=2a$ and $b+2=2a+1$, that is $b=2a-1$, $c=2a-3$, yielding the coefficients $\xi_2=\frac{1}{2}\xi_1^2$, $\xi_3=\frac{1}{8}\xi_1^3$ and $\xi_4=\left( \frac{\xi_3}{\xi_1}\right)^2=\frac{1}{64}\xi_1^4$. Therefore, we obtain the polynomial $$P(T)=1+\frac{1}{2}\xi_1T^{l_1}+\frac{1}{8}\xi_1^2T^{2l_1}, \ \ \ \  P(T)^2=1+\xi_1T^{l_1}+\frac{1}{2}\xi_1^2T^{2l_1}+\frac{1}{8}\xi_1^3T^{3l_1}+\frac{1}{64}\xi_1^4T^{4l_1}.$$

	\item Table \ref{Tab3}: In this case we have $P(T)^d=1+\xi'_1T^{l'_1}+\xi'_2T^{l'_2}+\xi'_3T^{l'_3}$ (hence exactly two among the $l_i$ are equal), with $\xi'_1,\xi'_2,\xi'_3$ such that exactly two of them belong to $x^{\mathbb{Z}}$ while the other one is a sum of two elements of $x^{\mathbb{Z}}$. Therefore, at least one between $\frac{1}{3}=\frac{\xi'_2}{(\xi'_1)^2}$ and $\frac{1}{27}=\frac{\xi'_3}{(\xi'_1)^3}=\frac{(\xi'_3)^2}{(\xi'_2)^2}$ must belong to $x^{\mathbb{Z}}$: this implies that $x$ is a power of $3$. Obviously, any solution of this form follows from a solution in base $3$, hence we can safely assume that $x=3$. But if at least two among $\xi'_1,\xi'_2,\xi'_3$ belong to $3^{\mathbb{Z}}$, from the relations written in the table we easily deduce that they must all belong $3^{\mathbb{Z}}$; since one of these is obtained as sum of two $\xi_i$ (which belong to $3^{\mathbb{Z}}$), we obtain the equation $3^a+3^b=3^c$, which has no solution for $a,b,c \in \mathbb{Z}$.
	
	\item Table \ref{Tab2}: In this last case $P(T)^d=1+\xi'_1T^{l'_1}+\xi'_2T^{l'_2}$, $d=2$ and $\xi'_2=\left(\frac{1}{2}\xi'_1\right)^2$. Hence, there are three ways in which the four coefficients $\xi_i$ (and the associated exponents $l_i$) can combine to form the two $\xi'_i$, namely:
	\begin{enumerate}
		\item $\xi'_1$ is obtained as a sum of three $\xi_i$, and $\xi'_2=\xi_j \in x^{\mathbb{Z}}$;
		
		\item $\xi'_2$ is obtained as a sum of three $\xi_i$, and $\xi'_1=\xi_j \in x^{\mathbb{Z}}$;
		
		\item $\xi'_1, \xi'_2$ are both obtained as a sum of two $\xi_i$ each.
	\end{enumerate}
	As $\xi_i \in x^{\mathbb{Z}}$, we can write $\xi_i=x^{\sigma_i}$, with $\sigma_i \in \mathbb{Z}$. Thus, these three cases give us (rearranging the indexes if needed) the following equations:
	\begin{enumerate}
		\item $\frac{1}{4}(x^{\sigma_1}+x^{\sigma_2}+x^{\sigma_3})^2=x^{\sigma_4}$, that is $ (x^{\sigma_1}+x^{\sigma_2}+x^{\sigma_3})^2=4x^{\sigma_4};$
		\item $\frac{1}{4}(x^{\sigma_1})^2=x^{\sigma_2}+x^{\sigma_3}+x^{\sigma_4}$, or equivalently that $ x^{2\sigma_1}=4x^{\sigma_2}+4x^{\sigma_3}+4x^{\sigma_4};$
		\item $\frac{1}{4}(x^{\sigma_1}+x^{\sigma_2})^2=x^{\sigma_3}+x^{\sigma_4}$, which yields $ x^{2\sigma_1}+x^{2\sigma_2}+2x^{\sigma_1+\sigma_2}=4x^{\sigma_3}+4x^{\sigma_4}.	$
	\end{enumerate}
	Let us study each case separately:
	\begin{enumerate}
		\item Assume without loss of generality that $\sigma_1 > \sigma_2 > \sigma_3.$ Then the left-hand side yields an integer whose base $x$ representation has at least two non-zero digits, one at place $2\sigma_3$ and one at place $2\sigma_1 \ge 2 \sigma_3 + 4$ or higher, while the one on right-hand side has either one or two digits (only if $x=3$), which are, however, consecutive: therefore this equation admits no solution.
		\item Similarly, the left-hand side of this equation yields an integer whose base $x$ representation has exactly one non-zero digit, whence the one on the right-hand side has at least two: again, this equation admits no solution.
		\item In this case, notice that if $x \ge 4$, the left-hand side yields an integer whose base $x$ representation has exactly three non-zero digits, while the right-hand side has exactly two: thus $x \in \{2,3\}$. 
		\begin{enumerate}
			\item If $x=3$, we have $3^{2\sigma_1}+3^{2\sigma_2}+2 \cdot 3^{\sigma_1+\sigma_2}=4\cdot 3^{\sigma_3}+4 \cdot 3^{\sigma_4};$ assume without loss of generality that $\sigma_1 > \sigma_2$ and $\sigma_3 > \sigma_4.$ Then $3^{2\sigma_1}+2 \cdot 3^{\sigma_1+\sigma_2}+3^{2\sigma_2}=3^{\sigma_3+1}+3^{\sigma_3}+3^{\sigma_4+1}+3^{\sigma_4};$ clearly, since the base $3$ representation of an integer must be unique, we must have
			$$\begin{cases} \sigma_3=\sigma_4+1 \\ 2\sigma_2=\sigma_4 \\ 2\sigma_1=\sigma_3+1  \end{cases} \text{ and this implies } \begin{cases} \sigma_1=\sigma_2+1 \\ \sigma_3=2\sigma_2+1 \\ \sigma_4=2\sigma_2  \end{cases}.$$
			Then, rearranging those $\sigma_i$ in increasing order, and computing the associated coefficients $\xi_i$, we obtain $$P(T)^2=1+\xi_1T^{l_1}+3\xi_1T^{l_1}+\xi_1^2T^{2l_1}+3\xi_1^2T^{2l_1}.$$

			\item If $x=2$ we obtain $2^{2\sigma_1}+2^{2\sigma_2}+ 2^{\sigma_1+\sigma_2+1}=2^{\sigma_3+2}+ 2^{\sigma_4+2};$ again, we can assume without loss of generality that $\sigma_1 > \sigma_2$ and $\sigma_3 > \sigma_4.$ These two base $2$ representation define the same integer, and thus must coincide. Thus we have
			$$\begin{cases} \sigma_1+\sigma_2+1=2\sigma_1 \\ 2\sigma_2=\sigma_4+2 \\ 2\sigma_1+1=\sigma_3+2  \end{cases} \text{ yielding } \begin{cases} \sigma_1=\sigma_2+1 \\ \sigma_3=2\sigma_2+1 \\ \sigma_4=2\sigma_2-2  \end{cases},$$
			which, rearranging the $\sigma_i$ in increasing order and computing the coefficients $\xi_i$, defines the solution $$P(T)^2=1+\xi_1T^{l_1}+2\xi_1T^{l_1}+\frac{1}{4}\xi_1^2T^{2l_1}+2\xi_1^2T^{2l_1}.$$

		\end{enumerate}
	\end{enumerate}
\end{itemize}

Therefore, the only solutions to the equation
$$P(T)^d=1+\xi_1T^{l_1}+\xi_2T^{l_2}+\xi_3T^{l_3}+\xi_4T^{l_4},$$ with $\xi_i \in x^{\mathbb{Z}}$ are the following:
\begin{enumerate}
	\item $x=3,d=3, P(T)^2=1+\xi_1T^{l_1}+3\xi_1T^{l_1}+\xi_1^2T^{2l_1}+3\xi_1^2T^{2l_1};$
	\item $x=2,d=2,P(T)^2=1+\xi_1T^{l_1}+\frac{1}{2}\xi_1^2T^{2l_1}+\frac{1}{8}\xi_1^3T^{3l_1}+\frac{1}{64}\xi_1^4T^{4l_1};$
	
	\item $x=2,d=2,P(T)^2=1+\xi_1T^{l_1}+2\xi_1T^{l_1}+\frac{1}{4}\xi_1^2T^{2l_1}+2\xi_1^2T^{2l_1}.$
\end{enumerate}

As before, our parametrization induces the correspondence $$y^d=1+x^{m_1}+x^{m_2}+x^{m_3}+x^{m_4} \mapsto P(T)=1+\xi_1T^{l_1}+\xi_2T^{l_2}+\xi_3T^{l_3}+\xi_4T^{l_4},$$


We can then go back to our two settings, and deduce the associated solutions of the equation $$y^d=1+x^{m_1}+x^{m_2}+x^{m_3}+x^{m_4}$$ via our parametrizations. 
\\
\underline{\textbf{First case}}: Remember that, in this setting, from our parametrizations we obtain the following correspondences $$y^d=1+x^{m_1}+x^{m_2}+x^{m_3}+x^{m_4} \mapsto z^d=1+x^{-n_1}+x^{-n_2}+x^{-n_3}+x^{-n_4}$$ $$\mapsto P(T)=1+\xi_1T^{l_1}+\xi_2T^{l_2}+\xi_3T^{l_3}+\xi_4T^{l_4}.$$

Therefore, each term of our polynomial is associated to a perfect power $x^{-n_i}$, which in turn will be used to compute the $x^{m_i}$. Let us study now the three solutions described before:

\begin{enumerate}
	\item $x=3,d=3, P(T)^2=1+\xi_1T^{l_1}+3\xi_1T^{l_1}+\xi_1^2T^{2l_1}+3\xi_1^2T^{2l_1}.$

	Hence there exists a permutation $\{n'_1,n'_2,n'_3,n'_4\}$ of the exponents $n_i$ such that
	$$\begin{cases} 3^{-n'_1} \mapsto \xi_1T^{l_1}, \\ 3^{-n'_2} \mapsto 3\xi_1T^{l_1}, \\ 3^{-n'_3} \mapsto \xi_1^2T^{2l_1}, \\ 3^{-n'_4} \mapsto 3\xi_1^2T^{2l_1}, \end{cases}\text{   which implies  } \begin{cases} n'_1 \in \mathbb{Z}^+, \\ -n'_2 = (-n'_1)+1, \\ -n'_3=2(-n'_1), \\ -n'_4=2(-n'_1)+1.  \end{cases}.$$
	Clearly, $n'_2 < n'_1 < n'_4 < n'_3$, thus the relations between the $n_i$ yield $n_1=n'_2$, $n_2=n'_1$, $n_3=n'_4$ and $n_4=n'_3$; thus, by substituting the values of $m_i$ and solving in function of $m_1$ we obtain	$$\begin{cases} m_2 \in \mathbb{Z}^+, \\ (m_4-m_2)= (m_4-m_3)+1, \\ (m_4-m_1)=2(m_4-m_3)+1, \\ m_4=2(m_4-m_3)+2,  \end{cases} \text{  yielding } \begin{cases} m_1=1, \\ m_2 \in \mathbb{Z}^+, \\ m_3=m_2+1, \\ m_4=2m_2.  \end{cases},$$
	giving the infinite family of solutions defined by $d=2$, $x=3$ and $$(m_1,m_2,m_3,m_4)=(1,m_2,m_2+1,2m_2), y=3^{m_2}+2.$$

	\item $x=2,d=2,P(T)^2=1+\xi_1T^{l_1}+\frac{1}{2}\xi_1^2T^{2l_1}+\frac{1}{8}\xi_1^3T^{3l_1}+\frac{1}{64}\xi_1^4T^{4l_1}.$

	Again, there is a permutation $\{n'_1,\dots,n'_4 \}$ of the exponents $n_i$ such that
	$$\begin{cases} 2^{-n'_1} \mapsto \xi_1T^{l_1}, \\ 2^{-n'_2} \mapsto \frac{1}{2}\xi_1^2T^{2l_1}, \\ 2^{-n'_3} \mapsto \frac{1}{8}\xi_1^3T^{3l_1}, \\ 2^{-n'_4} \mapsto \frac{1}{64}\xi_1^4T^{4l_1}, \end{cases}\text{  thus implying   } \begin{cases} n'_1 \in \mathbb{Z}^+, \\ -n'_2 = 2(-n'_1)-1, \\ -n'_3=3(-n'_1)-3, \\ -n'_4=4(-n'_1)-6.  \end{cases}$$
	
	This time we have $n'_1 < n'_2 < n'_3 < n'_4$, hence $n_i=n'_i$ for every $i$, and this system yields a linear system in the exponents $m_i$, which we can solve in function of $m_1$, obtaining
	$$\begin{cases} -(m_4-m_2)=-2(m_4-m_3)-1,  \\ -(m_4-m_1)=-3(m_4-m_3)-3, \\ -m_4=-4(m_4-m_3)-6,  \end{cases} \text{ and then } \begin{cases} m_2=2m_1-1,  \\ m_3=3m_1-3, \\ m_4=4m_1-6,   \end{cases},$$
	that gives the infinite family of solutions described by $d=2$, $x=2$ and $$
	(m_1,m_2,m_3,m_4)=(m_1,2m_1-1,3m_1-3,4m_1-6), y=1+2^{m_1-1}+2^{2m_1-3}.$$
	\item $x=2,d=2,P(T)^2=1+\xi_1T^{l_1}+2\xi_1T^{l_1}+\frac{1}{4}\xi_1^2T^{2l_1}+2\xi_1^2T^{2l_1}.$
	
	Thus, there exists a permutation $\{n'_1,n'_2,n'_3,n'_4\}$ of the exponents $n_i$ such that
	$$\begin{cases} 2^{-n'_1} \mapsto \xi_1T^{l_1}, \\ 2^{-n'_2} \mapsto 2\xi_1T^{l_1}, \\ 2^{-n'_3} \mapsto \frac{1}{4}\xi_1^2T^{2l_1}, \\ 2^{-n'_4} \mapsto 2\xi_1^2T^{2l_1}, \end{cases}\text{   which implies   } \begin{cases} n'_1 \in \mathbb{Z}^+, \\ -n'_2 = (-n'_1)+1, \\ -n'_3=2(-n'_1)-2, \\ -n'_4=2(-n'_1)+1.  \end{cases}$$
	This time we have $n'_2 < n'_1 < n'_4 < n'_3$, thus $n_1=n'_2$, $n_2=n'_1$, $n_3=n'_4$ and $n_4=n'_3$; again, by substituting these $n_i$ and solving in function of $m_1$ we get 	$$\begin{cases} m_2 \in \mathbb{Z}^+, \\ (m_4-m_2)= (m_4-m_3)+1, \\ (m_4-m_1)=2(m_4-m_3)-3, \\ m_4=2(m_4-m_3), \end{cases} \text{ which yields } \begin{cases} m_1=3, \\ m_2 \in \mathbb{Z}^+, \\ m_3=m_2+1, \\ m_4=2m_2-2,  \end{cases},$$
	giving the infinite family of solutions described by $d=2$, $x=2$ and $$(m_1,m_2,m_3,m_4)=(3,m_2,m_2+1,2m_2-2), y=2^{m_2-1}+3.$$
\end{enumerate}
We have thus proved the following result.
\begin{theorem}\label{5last}
	The only infinite families of solutions to the equation $y^d=1+x^{m_1}+x^{m_2}+x^{m_3}+x^{m_4},$ for $x,y,d$ positive integers and $\vect{m}=(m_1,m_2,m_3,m_4)$ such that $x,d \ge 2$, $m_1 < m_2 < m_3 < m_4$ and $m_3 \le C_3m_4,$ with $C_3 \in ]0,1[$ fixed, are the following:
	\begin{itemize}
		\item $x=3,d=2,\vect{m}=(1,m_2,m_2+1,2m_2), y=3^{m_2}+2;$
		\item $x=2,d=2,\vect{m}=(m_1,2m_1-1,3m_1-3,4m_1-6), y=1+2^{m_1-1}+2^{2m_1-3};$
		\item $x=2,d=2,\vect{m}=(3,m_2,m_2+1,2m_2-2), y=2^{m_2-1}+3.$ 
	\end{itemize}
\end{theorem}
\underline{\textbf{Second case}}:  In this case, the correspondence obtained is easier, since we get $$y^d=1+x^{m_1}+x^{m_2}+x^{m_3}+x^{m_4} \mapsto P(T)=1+\xi_1T^{l_1}+\xi_2T^{l_2}+\xi_3T^{l_3}+\xi_4T^{l_4}.$$
Since this time each term of $P(T)$ gives a perfect power $x^{m_i}$, we can immediately deduce the solutions:
\begin{enumerate}
	\item $x=2,d=2,(m_1,2m_1-1,3m_1-3,4m_1-6), y=1+2^{m_1-1}+2^{2m_1-3};$
	\item $x=3,d=2,(m_1,m_1+1,2m_1,2m_1+1), y=2\cdot 3^{m_1}+1;$
	\item $x=2,d=2,(m_1,m_1+1,2m_1-2,2m_1+1), y=1+2^{m_1-1}+2^{m_1}.$
\end{enumerate} 
Thus, we obtain the following result.
\begin{theorem}\label{5first}
	The only infinite families of solutions to the equation $y^d=1+x^{m_1}+x^{m_2}+x^{m_3}+x^{m_4},$ for $x,y,d$ positive integers and $\vect{m}=(m_1,m_2,m_3,m_4)$ such that $x,d \ge 2$, $m_1 < m_2 < m_3 < m_4$ and $m_1 \ge C_1m_4,$ with $C_1 \in ]0,1[$ fixed, are the following:
	\begin{itemize}
		\item $x=2,d=2,\vect{m}=(m_1,2m_1-1,3m_1-3,4m_1-6), y=1+2^{m_1-1}+2^{2m_1-3};$
		\item $x=3,d=2,\vect{m}=(m_1,m_1+1,2m_1,2m_1+1), y=2\cdot 3^{m_1}+1;$
		\item $x=2,d=2,\vect{m}=(m_1,m_1+1,2m_1-2,2m_1+1), y=1+2^{m_1-1}+2^{m_1}.$
	\end{itemize} 
\end{theorem}

\end{document}